%% file: MAIN.tex
\documentclass[11pt,sort&compress]{elsarticle}
\usepackage{hyperref,amssymb,amsmath,bm,mathtools,tikz,pgfplots,amsthm}
\usepackage[export]{adjustbox}
\usepackage[capitalise,nameinlink]{cleveref}
\usepackage{comment,enumitem,xcolor,array,float,subcaption,float}
\input{preamble.tex}
\pgfplotsset{compat=1.12}
\usepackage[margin=.6in]{geometry}
\graphicspath{}

\title{\LARGE Equivalence between a time-fractional and an integer-order gradient flow: 
	\\ The memory effect reflected in the energy}

\author[1]{Marvin Fritz}
\ead{fritzm@ma.tum.de}

\author[1]{Ustim Khristenko}
\ead{khristen@ma.tum.de}

\author[1]{Barbara Wohlmuth}
\ead{wohlmuth@ma.tum.de}

\cortext[cor1]{Corresponding author}
\address[1]{Department of Mathematics, Technical University of Munich, Germany}
\begin{document}

\begin{abstract}
	 Time-fractional partial differential equations are  nonlocal in time and show an innate memory effect. In this work, we propose an augmented energy functional which includes the history of the solution. Further, we prove the equivalence of a time-fractional gradient flow problem to an integer-order one based on our new energy. This equivalence guarantees the dissipating character of the augmented energy.  The state function of the integer-order gradient flow acts on an extended domain similar to the Caffarelli--Silvestre extension for the fractional Laplacian. Additionally, we apply a numerical scheme for solving time-fractional gradient flows, which is based on kernel compressing methods.  We illustrate the behavior of the original and augmented energy in the case of the Ginzburg--Landau energy functional.
\end{abstract}


\maketitle

\input{0_intro.tex}

\input{1_prelim.tex}
\input{2_analysis.tex}

\input{3_numerics.tex}

\input{4_simulation.tex}

\section*{Acknowledgments}
The authors gratefully acknowledge the support of the German Science Foundation (DFG) for funding part of this work through grant WO 671/11-1 and the European Union's Horizon 2020 research and innovation program under grant agreement No 800898.
\pagebreak 

\section*{References}
	\setlength{\bibsep}{0.75pt}
{\footnotesize
 \bibliography{literature}
\bibliographystyle{elsarticle-num}}

\end{document}

%% file: preamble.tex
\usepackage{mathtools,bm,comment,amssymb,enumitem,xcolor,array,float,subcaption,float,tikz,pgfplots,siunitx,booktabs,accents,mathrsfs}
\usepgfplotslibrary{groupplots}

\numberwithin{equation}{section}
\crefname{subsection}{Subsection}{Subsections}
\crefformat{equation}{(#2#1#3)}
\crefformat{enumi}{(#2#1#3)}
\crefname{figure}{Figure}{Figures}

\renewcommand{\tilde}[1]{\accentset{\sim}{#1}}
\newcommand{\til}[1]{\tilde{#1}}

\newcommand{\tu}{{\smash{\til{u}}}}
\newcommand{\tv}{{\smash{\til{v}}}}
\newcommand{\tE}{{\smash{\til{\mathcal{E}}}}}
\newcommand{\tH}{{\til{H}}}

\newcommand{\tX}{{\smash{\til{X}}}}
\newcommand{\lam}{\lambda}

\newcommand\G[1]{\mathcal{G}\{#1\}}

\newcommand{\weak}{\rightharpoonup}
\newcommand{\longweak}{\relbar\joinrel\rightharpoonup}
\newcommand{\eps}{\varepsilon}
\DeclareMathSymbol{\shortminus}{\mathbin}{AMSa}{"39}

\renewcommand{\d}{\textup{d}}

\newcommand{\dd}{\textup{d}}

\newcommand{\dt}{\,\textup{d}t}

\newcommand{\dx}{\,\textup{d}x}
\newcommand{\dth}{\,\textup{d}\theta}
\newcommand{\dl}{\,\textup{d}\lambda}
\newcommand{\ds}{\,\textup{d}s}
\newcommand{\ddt}{\frac{\dd}{\dd t}}

\newcommand{\p}{\partial}
\newcommand{\pt}{\p_t}

\newcommand{\pta}{\p_t^\alpha}
\newcommand{\grad}{{\nabla_{\!\!}}}

\newcommand{\R}{\mathbb{R}}

\newcommand{\I}{\mathcal{I}}
\newcommand{\X}{\mathcal{X}}

\renewcommand{\H}{\mathcal{H}}

\newcommand{\E}{\mathcal{E}}

\newcommand{\con}{\hookrightarrow}
\newcommand{\com}{\mathrel{\mathrlap{{\mspace{4mu}\lhook}}{\hookrightarrow}}}
\newcommand{\drv}[2]{\partial_{#1} #2} %

\newcommand{\norm}[1]{\left\lVert #1 \right\rVert}
\newcommand{\gradH}{\grad_H}
\newcommand{\gradHt}{\grad_\tH}

\newtheorem{remark}{Remark}

\newtheorem{lemma}{Lemma}
\newtheorem{example}{Example}
\newenvironment{contexample}{
	\addtocounter{example}{-1} \begin{example}[continued]}{
\end{example}}
\newtheorem{theorem}{Theorem}

\newtheorem{corollary}{Corollary}

\allowdisplaybreaks

\makeatletter
\@namedef{subjclassname@2020}{\textup{2020} Mathematics Subject Classification}
\makeatother

\let\originalleft\left
\let\originalright\right
\renewcommand{\left}{\mathopen{}\mathclose\bgroup\originalleft}
\renewcommand{\right}{\aftergroup\egroup\originalright}

\makeatletter
\newcases{lrdcases}
{\quad}
{$\m@th\displaystyle{##}$\hfil}
{$\m@th\displaystyle{##}$\hfil}
{\lbrace}
{\rbrace}
\makeatother

\newcolumntype{C}{>{\centering\arraybackslash}m{10em}}

\newcommand{\abs}[1]{\left| #1 \right|}

\definecolor{color0}{rgb}{0.12156862745098,0.466666666666667,0.705882352941177}
\definecolor{color1}{rgb}{1,0.498039215686275,0.0549019607843137}
\definecolor{color2}{rgb}{0.172549019607843,0.627450980392157,0.172549019607843}
\definecolor{color3}{rgb}{0.83921568627451,0.152941176470588,0.156862745098039}
\definecolor{color4}{rgb}{0.580392156862745,0.403921568627451,0.741176470588235}

%% file: 0_intro.tex
\section{Introduction}
In this work, we investigate the influence of the history on the energy functional of time-fractional gradient flows, i.e., the standard time derivative is replaced by a derivative of fractional order in the sense of Caputo. By definition, the system becomes nonlocal in time, and  the history of the state function plays a significant role in its time evolution. Recently, time-fractional partial differential equations (FPDEs) became of increasing interest. Their innate memory effect appears in many applications, e.g., in the mechanical properties of materials \cite{torvik1984appearance}, in viscoelasticity \cite{mainardi2010fractional} and -plasticity \cite{diethelm1999solution}, in image \cite{cuesta2011some} and signal processing \cite{marks1981differintegral}, in diffusion \cite{olmstead1976diffusion} and heat progression problems \cite{podlubny1995application}, in the modeling of solutes in fractured media \cite{benson2000application}, in combustion theory \cite{audounet1998threshold}, bioengineering \cite{magin2006fractional}, damping processes \cite{gaul1991damping}, and even in the modeling of love \cite{ahmad2007fractional} and happiness \cite{song2010dynamical}.

The theory of gradient flows is well-investigated in the integer case, e.g., see the book \cite{ambrosio2008gradient}  and the work \cite{otto2001geometry} regarding the analysis of the porous medium equation as a gradient flow. One of the most important properties of a gradient flow is its energy dissipation, which can be immediately derived from the variational formulation and by the chain rule.
This relation is also called the principle of steepest descent.
Typical applications are the heat equation with the underlying Dirichlet energy, and the Ginzburg--Landau energy, which results  in the well-known Cahn--Hilliard \cite{miranville2019cahn} and Allen--Cahn equation \cite{allen1972ground} depending on the choice of the underlying Hilbert space. We also mention the Fokker--Planck \cite{jordan1998variational}, and the Keller--Segel equations \cite{blanchet2013gradient}, which can be written and analyzed as gradient flows.

Some of their time-fractional counterparts have been investigated in the literature, e.g., the time-fractional gradient flows of type Allen--Cahn \cite{du2019time}, Cahn--Hilliard \cite{fritz2020time}, Keller--Segel \cite{kumar2017new}, and Fokker--Planck \cite{duong2019wasserstein} type. Up to now there is no unified theory for time-fractional gradient flows, and it is not yet known whether the dissipation of energy is fulfilled, see also the discussions in \cite{zhang2020non,liu2020fast,liao2020second,chen2019accurate}. From a straightforward testing of the variational form as in the integer order setting, one can only bound the energy by its initial state but one cannot say whether it is dissipating continuously in time. Several papers investigated the dissipation law of time-fractional phase field equations numerically and proposed weighted schemes in order to fulfill the dissipation of the discrete energy, see \cite{quan2020define,quan2020numerical,tang2019energy,ran2021implicit,zhang2020high,liang2020lattice,ji2020simple,ji2020linear}.

Our main contribution is the well-posedness of time-fractional gradient flows and the introduction of a new augmented energy, which is motivated by the memory structure of time-fractional differential equations and therefore, includes an additional term representing the history of the state function. We show that the integer-order gradient flow  corresponding to this augmented energy on an extended Hilbert space is equivalent to the original time-fractional model. Consequently, the augmented energy is monotonically decreasing in time. We note that the state function of the augmented gradient flow acts on an extended domain similar to the Caffarelli--Silvestre approach \cite{caffarelli2007extension} of the fractional Laplacian using harmonic extensions. This technique of dimension extension has also be used in the analysis of random walks \cite{molchanov1969symmetric} and embeds a long jump random walk to a space with one added dimension. %

In \cref{Sec:Pre}, we state some preliminary results on fractional derivatives and Bochner spaces. Moreover, we state and prove a theorem of well-posedness of fractional gradient flows. We state the main theorem of the equivalence of the fractional and the extended gradient flows in \cref{Sec:Main} and give a complete proof. Afterwards, we give two corollaries, one stating the consequence of energy dissipation and the other concerning the limit case $\alpha=1$ in the fractional order. In \cref{Sec:Num}, we present an algorithm to solve the time-fractional system based on rational approximations. Lastly, we illustrate in \cref{Sec:Sim} some simulations in order to show the influence of the history energy.

%% file: 1_prelim.tex
\section{Analytical Preliminaries and Well-Posedness of Time-Fractional Gradient Flows} \label{Sec:Pre}

In the following, let $H$ be a separable Hilbert space and $X$ a Banach space such that it holds 
\begin{gather*} X \com H  \con X', 
\end{gather*}
where the embedding $\con$ is continuous and dense, and $\com$ is additionally compact. We apply the Riesz representation theorem to identify $H$  with its dual. %
In this regard, $(X,H,X')$ forms a Gelfand triple, e.g.,
$$ \begin{aligned}
\big(H^k(\Omega), H^{k-j}(\Omega), H^{-k}(\Omega)\big) \text{ with } k\geq j>0. %
\end{aligned}$$                                         
The duality pairing in $X$ is regarded as a continuous extension of the scalar product of the Hilbert space $H$ in the sense
\begin{equation*}\label{eq:innergelfand} \begin{aligned}
	\langle u,v \rangle_{X' \times X} &= (u,v)_H &&\quad \forall u \in H, v \in X, %
\end{aligned} \end{equation*}

We call a function Bochner measurable if it can be approximated by a sequence of Banach-valued simple functions and consequently, we define the Bochner spaces $L^p(0,T;X)$ as the equivalence class of Bochner measurable functions $u:(0,T) \to X$ such that $t \mapsto \|u(t)\|_X^p$ is Lebesgue integrable. Note that $L^2(0,T;X)$ is a Hilbert space if $X$ is a Hilbert space, e.g., see \cite{diestel1977vector}. The Sobolev--Bochner space $W^{1,p}(0,T;X)$ consists of functions in $L^p(0,T;X)$ such that their distributional time derivatives are induced by functions in $L^p(0,T;X)$. 

\subsection{Fractional derivative}
Let us introduce the linear continuous Riemann--Liouville integral operator $\I_\alpha \in  \mathscr{L}(L^1(0,T;X))$ of order $\alpha \in (0,1)$ of a function $u\in L^1(0,T;X)$, defined by 
\begin{equation} \label{Def:RLintegral}
	\I_{\alpha} u := g_{\alpha} * u,
\end{equation}
where the singular kernel $g_{\alpha}\in L^1(0,T)$ is given by $g_{\alpha}(t)= t^{\alpha-1}/\Gamma(\alpha)$, and the operator $*$ denotes the convolution on the positive half-line with respect to the time variable.
Note that the operator $\I_\alpha$ has a complementary element in the sense
\begin{equation} \label{Eq:InverseKernel}
\I_\alpha \I_{1-\alpha} u =\I_1 u = 1 * u,
\end{equation}
see \cite{diethelm2010analysis}.
Then, the fractional derivative of order $\alpha \in (0,1)$ in the sense of Caputo 
is defined by
\begin{equation}\label{Eq:DefFracDer}
	\pta u  := g_{1-\alpha} * \pt u = \I_{1-\alpha} \pt u,
\end{equation}
see, e.g., \cite{diethelm2010analysis,kilbas2006theory}.
In the limit cases $\alpha=0$ and $\alpha=1$, we define $\p_t^0 u = u-u_0$ and  $\p_t^1 u = \p_t u$, respectively.
One can write~\cref{Eq:DefFracDer} as
\begin{equation*}
\p_t^\alpha u(t)= \frac{1}{\Gamma(1-\alpha)}  \int_0^t \frac{\p_s u(s)}{(t-s)^{\alpha}} \ds,
\end{equation*}
in $X$ for a.e. $t \in (0,T)$.
The infinite-dimensional valued integral is understood in the Bochner sense. We note that the Caputo derivative requires a function which is absolutely continuous. But this definition can be generalized to a larger class of functions which coincide with the classical definition in case of absolutely continuous functions, see \cite{li2018some,fritz2020time}. 

Similar to before, we define the fractional Sobolev--Bochner space $W^{\alpha,p}(0,T;X)$ as the functions in $L^p(0,T;X)$ such that their $\alpha$-th fractional time derivative is in $L^p(0,T;X)$.
 Let us remark that by \cref{Eq:InverseKernel} it follows
\begin{equation} 
	(\I_{\alpha} \p_t^\alpha u)(t) =\I_{\alpha}\I_{1-\alpha} \pt u = \I_1\p_t u(t) = u(t)-u_0.
	\label{Eq:InverseConvolution}
\end{equation}
As in the integer-order setting, there are continuous and compact embedding results \cite{ouedjedi2019galerkin,wittbold2020bounded,zacher2009weak,li2018some}. In particular, provided that $X$ is compactly embedded in $H$, it holds \begin{equation} \begin{aligned} \label{eq:aubin} 
W^{\alpha,p'}(0,T;X') \cap L^p(0,T;X) &\con C([0,T];H), \quad \frac{1}{p}+\frac{1}{p'}=1, \quad \alpha>0, \\W^{\alpha,1}(0,T;X) \cap L^p(0,T;X') &\com L^r(0,T;H), \quad 1\leq r<p,\quad \alpha>0.\end{aligned}\end{equation} 
Moreover, it holds the following version of the Gr\"onwall--Bellman inequality in the fractional setting.
\begin{lemma}[{cf. \cite[Corollary 1]{fritz2020time}}]
	\label{Lem:FractionalGronwall1}
	Let $w,v \in L^1(0,T;\R_{\geq 0})$, and $a,b\ge0$. If $w$ and $v$ satisfy the inequality
$$
		w(t)+g_{\alpha}*v(t) \leq a + b\;(g_\alpha * w)(t) \qquad \text{a.e. }  t \in (0,T),
$$
	then it holds
	$w(t)+ v(t) \leq a \cdot C(\alpha,b,T)$ for almost every $t \in (0,T)$.
\end{lemma}

We mention the following lemma which provides an alternative to the classical chain rule $\ddt f(u) = f'(u) \ddt u$ to the fractional setting for $\lambda$-convex (or semiconvex) functionals $f:X \to \R$ with respect to $H$, i.e., $x \mapsto f(x)-\frac{\lambda}{2}\|x\|_H^2$ is convex for some $\lambda\in \R$. 
If $f$ is twice differentiable and $\lambda=-1$, then semiconvexity implies $f''(x) \geq -1$ which is also called dissipation property of $f'$. The result of the fractional chain inequality for the quadratic function $f(u)=\frac12 \|u\|_H^2$ is well-known in case of $u \in H^1(0,T;H)$, see \cite[Theorem 2.1]{vergara2008lyapunov}, saying
\begin{equation} \label{Eq:ChainSquared} \frac12 \pta \|u\|_H^2 \leq (\pta u,u)_H \quad \forall u \in H^1(0,T;H).
\end{equation}
It has been generalized to convex functionals $f$ in \cite[Proposition 2.18]{li2018some} in the form of inequality
$$\pta f(u) \leq \langle f'(u), \pta u\rangle_{X' \times X} \quad \forall u \in C^1([0,T);X),$$
and applying it to the convex functional $x \mapsto f(x)-\frac{\lambda}{2}\|x\|_H^2$ directly gives the following result for semiconvex functionals.

\begin{lemma}\label{Lem:Chain} Let $H$ be a Hilbert space, $X \con H$ a Banach space, and
	 $u \mapsto f(u) \in \R$ a Fr\'echet differentiable functional on $X$. If $f$ is $\lambda$-convex for some $\lambda\in \R$ with respect to $H$, then it holds 
	$$\pta f(u) \leq \langle  f'(u),\pta u \rangle_{X'\times X} +  \frac{\lambda}{2} \pta \|u\|_H^2 -\lambda \langle \pta u,u \rangle_{X' \times X} \qquad \forall u \in C^1([0,T);X). $$
\end{lemma}
Wwe note that in the discrete setting the required regularity $u \in C^1([0,T);X)$ is often satisfied. However, this regularity is not necessarily available for weak solutions. Here, we refer to \cite[Proposition 1]{fritz2020time} for the convolved version
\begin{equation} \label{Eq:ChainIneq} f(u(t))- f(u_0) \leq \big(g_{\alpha}*\langle f'(u),\pta u\rangle_{X' \times X}\big)(t)   + \frac{\lambda}{2} \big(\|u\|^2_{H}-\|u_0\|^2_{H}\big)- \lambda \big( g_{\alpha}*\langle \p_t^\alpha u,u \rangle_{X'\times X}\big)(t), \end{equation} 
for a.e. $t \in (0,T)$ which requires  $u \in W^{\alpha,p'}(0,T;X) \cap L^p(0,T;X)$, $u_0 \in H$, and $f'(u) \in L^{p'}(0,T;X')$.

\subsection{Time-fractional gradient flows in Hilbert spaces} \label{Sec:TFGF}
In this work, we focus on the time-fractional gradient flow in the Hilbert space $H$, defined as the variational problem
 \begin{equation} \label{Eq:GradientVariational}
\pta (u,v)_H + \delta \E(u,v)=0 \quad \forall v \in X,
 \end{equation}
for a given nonlinear energy functional~$\E:X \to \R$, where $\delta \E:X \times X \to \R$ denotes its G\^ateaux derivative:
\begin{equation*}
\delta\E(u,v) 
:= \lim\limits_{h\to 0}\frac{\E(u+hv)-\E(u)}{h}
\qquad \forall u,v \in X.
\end{equation*}
We also define the gradient of $\E$ in the Hilbert space~$H$ as $\gradH\E:X \to X'$ such that at $u\in X$ it holds
\begin{equation*}\label{key}
\langle \gradH\E(u),v \rangle_{X' \times X} = \delta \E(u,v)
\qquad \forall v \in X.
\end{equation*}
Then,~\eqref{Eq:GradientVariational} can be equivalently written as $\pta u = - \gradH\E(u)$ in $X'$ or
\begin{equation*} \label{Eq:GradientVariational2} 
\langle \pta u + \gradH\E(u),\, v \rangle_{X'\times X} =0 \quad \forall v \in X.
\end{equation*}
Moreover, we equip this variational problem with the initial data $u_0 \in H$. In general, we do not have $u \in C([0,T];H)$ and therefore, we do not have $u(t) \to u_0$ in $H$ as $t \to 0$.  Instead, we are going to prove $g_{1-\alpha} * (u-u_0) \in C([0,T];H)$ and the initial data is satisfied in the sense $g_{1-\alpha} * (u-u_0)(t)\to 0$ in $H$ as $t\to 0$. Moreover due to the Sobolev embedding theorem, we can actually prove $u-u_0 \in C([0,T];X')$ for $\alpha>1/p$. Consequently, for such values of $\alpha$ the initial is satisfied in the sense $u(t) \to u_0$ in $X'$.

 \begin{example}\label{Ex:1}
 	We consider the energy functional
 	\begin{equation}\label{eq:Energy}
 	\E(u)=\int_\Omega \big( f(u(x)) + \frac12 |\nabla u(x)|^2 \big) \dx,
 	\end{equation} 
 	for some $f \in C^1(\R;\R_{\geq 0})$. Choosing the double-well function $f(u)=(1-u^2)^2$, the energy corresponds to the Ginzburg--Landau energy \cite{ginzburg1963frictional} with $\eps=1$, and selecting $f(u)=0$ reduces to the Dirichlet energy. In order to justify the well-definedness of the second term of the integral, we require $X \subset H^1(\Omega)$. In case of the double-well function, we additionally require $X \subset H^1(\Omega)\cap L^4(\Omega)$.
 	
 	Let us consider the Sobolev space with zero mean
 	\begin{equation*}
 	\dot H^1(\Omega)=\{u \in H^1(\Omega): (u,1)_{L^2(\Omega)}=0\},
 	\end{equation*}
 	equipped with the scalar product $(\nabla \cdot, \nabla \cdot)_{L^2(\Omega)}$, which is equivalent to the inherited one on $H^1(\Omega)$ by the Poincar\'e inequality \cite{evans2010partial}. Moreover, we equip its dual space $\dot H^{-1}(\Omega)=(\dot H^1(\Omega))'$ with the graph norm $\|\nabla (-\Delta)^{-1} \cdot\|_{L^2(\Omega)}$, which is equivalent to the standard dual norm, see \cite[Remark 2.7]{miranville2019cahn}. Here, homogeneous Neumann boundary conditions are associated with the Laplace operator.
 	
 	Then, the G\^ateaux derivative of the energy functional~\eqref{eq:Energy} can be written using scalar products of the Hilbert spaces $H \in \{\dot H^1(\Omega), L^2(\Omega),  \dot H^{-1}(\Omega)\}$ as follows:
 	\begin{equation*} \label{Eq:Gateaux}
		\delta \E(u,v)=((-\Delta)^{-1} f'(u)+u,v)_{\dot H^1(\Omega)}=(f'(u)-\Delta u,v)_{L^2(\Omega)}=(-\Delta f'(u)+\Delta^2 u,v)_{\dot H^{-1}(\Omega)} 
	\end{equation*}	
	for all $u,v \in X$, assuming that $X$ is regular enough.
	In particular, if $X=C_c^\infty(\Omega)$, the strong form of~\eqref{Eq:GradientVariational} results in the respective gradient flows:  
 	\begin{align*}
 	\pta u &=\Delta^{-1} f'(u) - u, &&\hspace{-12em} \text{ when }H=H^1(\Omega),  \\
 	\pta u &= \Delta u - f'(u), &&\hspace{-12em}\text{ when }H=L^2(\Omega),  \\
 	\pta u &= \Delta( f'(u) -\Delta u), &&\hspace{-12em}\text{ when }H=\dot H^{-1}(\Omega).
 	\end{align*}
 	In the case of $f=0$, it results in a fractional ODE, in the fractional heat equation, and in the fractional biharmonic equation, respectively. For a double-well potential it yields the Allen--Cahn equation in case of $H=L^2(\Omega)$ and the Cahn--Hilliard equation for $H=\dot H^{-1}(\Omega)$.%
\end{example}

 \subsection{Well-posedness of time-fractional gradient flows} We provide the following proposition which yields the existence of variational solutions to time-fractional gradient flows.  In order to show uniqueness and continuous dependence on the data, we have to assume that $\E$ is additionally semiconvex, see \cref{th:prelim:Existence2} below.
 In \cite[Chapter 5]{vergara2006convergence}, similar results are proven for the exponential kernel $k_{1-\alpha}(t)=C e^{-\omega t} t^{-\alpha}$ with some constants $C,\omega>0$.  
 
 \begin{theorem}\label{th:prelim:Existence}
 	Let  $X$ be a separable, reflexive Banach space that is compactly embedded in the separable Hilbert space $H$. Further, let $u_0 \in H$, $\E \in C^1(X;\R)$,  and $\delta \E:X \times X \to \R$ be bounded and semicoercive in the sense %
\begin{align}
 	\|\gradH\E(u)\|_{X'} &\leq C_0 \|u\|_X^{p-1} + C_0, 
 	\label{eq:continuity}\\
 	\delta \E(u,u) &\geq C_1 \|u\|_X^p - C_2 \|u\|_H^2,
 	\label{eq:coercivity}
\end{align}
	for all $u \in X$ for some positive constants $C_0,C_1,C_2<\infty$ and $p>1$.  Moreover, we assume that the realization of $\gradH\E$ as an operator from $L^p(0,T;X)$ to $L^{p'}(0,T;X')$ is weak-to-weak continuous.
 	Then the time-fractional gradient flow $\pta u = - \gradH\E(u)$ with $\alpha \in (0,1]$ admits a variational solution in the sense that 
 	 	 \begin{equation*} \begin{aligned}
 	u &\in L^p(0,T;X) \cap L^\infty(0,T;H) \cap W^{\alpha,p'}(0,T;X'),  \\  g_{1-\alpha} * (u-u_0) &\in C([0,T];H),\end{aligned}\end{equation*} 
 	fulfills the variational form
 	\begin{equation*}
 	 \pta (u,v)_H + \delta \E(u,v) =0 \quad \forall v \in X,
 	\end{equation*}
and the energy inequality
\begin{equation} \label{Eq:EnergyFinal2}
	 \|u\|_{L^p(0,T;X)}^p + \|u\|_{L^\infty(0,T;H)}^2  + \|\gradH\E(u)\|_{L^{p'}(0,T;X')}^{p'} \leq C(T) \cdot \big(1+\|u_0\|_H^2\big). 
\end{equation}
 \end{theorem}

\begin{proof}
We employ the Faedo--Galerkin method \cite{lions1969some} to reduce the time-fractional PDE to an fractional ODE, which admits a solution $u_k$ due to the well-studied theory given in \cite{diethelm2010analysis,kilbas2006theory}. We derive energy estimates, which imply the existence of weakly convergent subsequences by the Eberlein--\v{S}mulian theorem \cite{brezis2010functional}. We pass to the limit $k\to \infty$ and apply compactness methods to return to the variational form of the time-fractional gradient flow. Recently, the Faedo--Galerkin method has been applied to various time-fractional PDEs, see, e.g., \cite{li2018some,fritz2020subdiffusive,djilali2018galerkin,fritz2020time,kubica}.
\medskip 

\noindent \textit{Discrete approximation.}
Since $X$ is a separable Banach space, there is a finite-dimensional dense subspace of $X$, called $X_k$, which is spanned by the $k \in \mathbb{N}$ elements $x_1,\dots,x_k \in X$. We consider the equation 
\begin{equation} (\pta u_k, v_k )_H + \delta \E(u_k,v_k) = 0 \quad \forall v_k \in X_k. \label{Eq:Faedo}
\end{equation}
Hence, we are looking for a function $u_k=\sum_{j=1}^k u_k^j x_j$ such that the fractional differential equation system
\begin{equation*}
\sum_{j=1}^k \pta u_k^j (x_j, x_n)_H + \big\langle \gradH\E(\text{$\textstyle\sum_{j=1}^k u_k^j x_j$}), x_n\big\rangle_{X' \times X}=0 \quad \forall n \in \{1,\dots,k\},
\end{equation*}
with initial data $u_k(0)=\Pi_k u_0$ is fulfilled. Here, $\Pi_k:X \to X_k$ denotes the orthogonal projection onto $X_k$. This fractional ODE admits an absolutely continuous solution vector $(u_k^1,\dots,u_k^k)$ since $u \mapsto \E(u)$ is assumed to be continuously differentiable, see \cite{diethelm2010analysis}. Therefore, the local existence of $u_k \in AC([0,T_k];X_k)$ is ensured. If we can derive an uniform bound on $u_k$, we can extend the time interval by setting $T_k=T$. Moreover, if $\gradH$ is assumed to be Lipschitz continuous, then we can argue by a blow-up alternative to achieve global well-posedness with $T=\infty$, see the discussion in \cite{djilali2018galerkin}.%
\medskip 

\noindent \textit{Energy estimates.} Taking the test function $v_k=u_k$ in \cref{Eq:Faedo} gives %
\begin{equation*}
(\pta u_k,u_k)_H + \delta \E(u_k,u_k)= 0.
\end{equation*}
Applying the fractional chain inequality $\frac12\pta  \|u_k\|_H^2 \leq (\pta u_k,u_k)_H$, see \cref{Eq:ChainSquared}, and the semicoercivity of $\delta\E$, see \cref{eq:coercivity}, gives the estimate
\begin{equation*}
\frac12 \pta \|u_k\|_H^2 + C_1\|u_k\|_X^p  \leq C_2\|u_k\|_H^2.
\end{equation*}
Taking the convolution with $g_\alpha$ of this estimate yields with the identity \cref{Eq:InverseConvolution}
\begin{equation} \label{Eq:PreGron}
\| u_k(t)\|_H^2 + C_1 \big(g_\alpha * \|u_k\|_X^p\big)(t)  \leq \|u_k(0)\|_H^2 + C_2 \big(g_\alpha * \|u_k\|_H^2\big)(t).
\end{equation}
We apply the inequality 
\begin{equation} \label{Ineq:Aux} \int_0^{t} \|u_k(s)\|_X^p \ds \leq T^{1-\alpha} \int_0^{t}  (t-s)^{\alpha-1} \|u_k(s)\|_X^p \ds,
	\end{equation}
and the fractional Gr\"onwall--Bellman inequality, see \cref{Lem:FractionalGronwall1}, to \cref{Eq:PreGron}, and find
\begin{equation} \label{Eq:BoundGron}
\|u_k\|_{L^\infty(0,T_k;H)}^2  +\|u_k\|_{L^p(0,T_k;X)}^p \leq C(T) \| u_0\|_H^2. 
\end{equation}

This gives the uniform boundedness of the sequence $u_k$ in the spaces $L^\infty(0,T_k;H)$ and $L^p(0,T_k;X)$. Therefore, we can extend the time interval by setting $T_k=T$. The boundedness assumption of $\delta\E$, see \cref{eq:continuity}, immediately  gives the uniform bound 
\begin{equation} \label{Eq:BoundGron2}
	\|\gradH \E(u_k)\|_{L^{p'}(0,T;X')}^{p'} \leq C(T) + C \int_0^T  \|u_k(t)\|_X^{p'(p-1)} \dt \leq C(T) \cdot \big(1+\|u_0\|_H^2\big).
\end{equation}

\noindent \textit{Estimate on the fractional time-derivative.}
Taking an arbitrary function $v \in L^p(0,T;H)$ and testing with its projection onto $X_k$ in \cref{Eq:Faedo} gives a bound of $\pta u_k$ in $L^{p'}(0,T;X')$ due to the boundedness assumption of $\E$. Indeed, let $v\in L^p(0,T;X)$ and denote $\Pi_k v = \sum_{j=1}^k v_j^k x_j$ for time-dependent coefficient functions $v_j^k : (0,T) \to \R$, $j \in \{1,\dots,k\}$. We multiply equation \cref{Eq:Faedo} by $v_j^k$, take the sum from $j=1$ to $k$, and integrate over the interval $(0,T)$, giving
\begin{equation*}
\begin{aligned}\left|\int_0^T (\p_t^\alpha u_k, v)_H\dt\right| &= \left|\int_0^T \delta \E(u_k,\Pi_k v) \dt\right|  \leq \|\delta \E(u_k,\cdot)\|_{L^{p'}(0,T;X')}  \|v\|_{L^p(0,T;X)},
\end{aligned}
\end{equation*}
where we used that $(\pta u_k,v)_H=(\pta u_k, \Pi_k v)_H$ due to the invariance of the time derivative under the adjoint of the projection operator.
Therefore, we have 
\begin{equation} \label{Eq:BoundTime}
\| \p_t^\alpha u_k \|_{L^{p'}(0,T;X')} = \sup_{\|\varphi \|_{L^{p}(0,T;X)} \leq 1} \Big| \int_0^T \langle \pta u_k,\varphi  \rangle_{X' \times X} \dt \Big|  \leq C(T) \cdot \big(1+\|u_0\|_H^2\big).
\end{equation}

\noindent \textit{Limit process.} The bounds \cref{Eq:BoundGron}--\cref{Eq:BoundTime} give the energy inequality
\begin{equation} \label{Eq:BoundEnergy}
\|u_k\|_{L^\infty(0,T;H)}^2  +\|u_k\|_{L^p(0,T;X)}^p + \|\pta u_k\|_{L^{p'}(0,T;X')}^{p'} +  \|\gradH\E(u_k)\|_{L^{p'}(0,T;X')}^{p'} \leq C(T) \cdot \big(1+\|u_0\|_H^2\big),
\end{equation}
which implies the existence of weakly/weakly-$*$ convergent subsequenes by the Eberlein--\v{S}mulian theorem \cite{brezis2010functional}. By a standard abuse of notation, we drop the subsequence index. Hence, there exist limit functions $u$ and $\xi$ such that for $k \to \infty$
\begin{equation} \label{Eq:Weak} \begin{aligned}	
u_k &\longweak u &&\text{weakly-$*$ in } L^\infty(0,T;H), \\
u_k &\longweak u &&\text{weakly\phantom{-*} in } L^p(0,T;X), \\
\p_t^\alpha  u_k &\longweak \p_t^\alpha  u &&\text{weakly\phantom{-*} in } L^{p'}(0,T;X'), \\
u_k &\longrightarrow u &&\text{strongly\hspace{1mm} in }
L^{p'}(0,T;H), \\
\gradH \E(u_k) &\longweak \xi &&\text{weakly\phantom{-*} in } L^{p'}(0,T;X'). %
\end{aligned} \end{equation} 
Here, we applied the fractional Aubin--Lions compactness lemma \cref{eq:aubin} to achieve the strong convergence of $u_k$ in $L^{p'}(0,T;H)$. Moreover, we concluded that the weak limit of $\p_t^\alpha u_k$  is equal to  $\p_t^\alpha u$, see \cite[Proposition 3.5]{li2018some}. %

In the last step we take the limit $k\to \infty$ in the variational form and use that $\Pi_k X_k$ is dense in $X$. By multiplying the Faedo--Galerkin system \cref{Eq:Faedo} by a test function $\eta \in C^\infty_c(0,T)$ and integrating over the time interval $(0,T)$, we find
\begin{equation} \begin{aligned}
\int_0^T \langle \p_t^\alpha u_k,x_j\rangle_{X' \times X} \eta(t) \dt   + \int_0^T \delta \E(u_k,x_j) \eta(t) \dt =0.
\end{aligned} \label{Eq:FaedoTime}
\end{equation}
for all $j \in \{1,\dots,k\}$. We pass to the limit $k \to \infty$ and note that the functional
\begin{equation*}u_k \mapsto  \int_0^T \langle \pta u_k, x_j\rangle_{X' \times X} \eta(t) \dt, \end{equation*}
is linear and continuous on $L^{p'}(0,T;X')$, since we have by the H\"older inequality
\begin{equation*}\left| \int_0^T \langle\pta u_k, x_j\rangle_{X' \times X} \eta(t) \dt \right| \leq  \|\pta u_k\|_{L^{p'}(0,T;X')} \|x_j\|_X \|\eta\|_{L^p(0,T)}. \end{equation*}
The weak convergence \cref{Eq:Weak} gives by definition as $k \to \infty$
\begin{equation*} \int_0^T \langle \pta u_k,x_j \rangle_{X' \times X} \eta(t) \dt \longrightarrow  \int_0^T \langle \pta u,x_j \rangle_{X' \times X} \eta(t) \dt. \end{equation*}

It remains to treat the integral involving the energy term. Since the realization $\gradH\E :L^p(0,T;X) \to L^{p'}(0,T;X')$ is weak-to-weak continuous by assumption, we have by the weak convergence $u_k \weak u$ in $L^p(0,T;X)$ also $\gradH \E(u_k) \weak \gradH\E(u)$ weakly in $L^{p'}(0,T;X')$ and therefore, $\xi=\gradH\E(u)$ in $L^{p'}(0,T;X')$.
Applying this weak convergence to the second term in \cref{Eq:FaedoTime} completes the limit process. Indeed, taking $k \to \infty$ in \cref{Eq:FaedoTime} and using the density of $\cup_k X_k$ in $X$, we get
\begin{equation*} \begin{aligned}
\int_0^T \langle\p_t^\alpha u,v\rangle_{X' \times X} \eta(t) \dt  + \int_0^T \delta \E(u,v) \eta(t) \dt =0, \qquad \forall v \in X,\, \eta \in C_c^\infty(0,T).
\end{aligned} 
\end{equation*}
Applying the fundamental lemma of calculus of variations, we finally find
\begin{equation*} \begin{aligned}
\p_t^\alpha (u,v)_H +\delta \E(u,v)  =0, \qquad \forall v \in X.
\end{aligned} 
\end{equation*}

\noindent \textit{Initial condition.} From the estimate above, we have $\p_t^\alpha u_k\in L^{p'}(0,T;X')$. The definition of the fractional derivative then yields
\begin{equation*} 
g_{1-\alpha}*u_k\in L^p(0,T;X) \cap W^{1,p'}(0,T;X') \hookrightarrow C^0([0,T];H),
\end{equation*} see the embedding \cref{eq:aubin}. Therefore, it holds $g_{1-\alpha}*(u_k-u_0) \in C^0([0,T];H)$ and  the given initial $u_0$ is satisfied in the sense 
$g_{1-\alpha}*(u_k-u_0) \to 0$ in $H$ as $t\to 0$. %
\medskip 

\noindent \textit{Energy inequality.} We prove that the solution $u$ satisfies the energy inequality. First, we note that norms are weakly/weakly-$*$ lower semicontinuous, e.g., we have $u_k \weak u$ in $L^p(0,T;X)$ and therefore, we infer
$$\| u \|_{L^p(0,T;X)} \leq \liminf_{k \to \infty} \|u_k\|_{L^p(0;T;X)}.$$  
Hence, from the discrete energy inequality \cref{Eq:BoundEnergy} and the weak convergence
\begin{equation*} \label{Eq:BoundEnergyCont}
\|u\|_{L^\infty(0,T;H)}^2  +\|u\|_{L^p(0,T;X)}^p + \|\pta u\|_{L^{p'}(0,T;X')}^{p'} +  \|\gradH\E(u)\|_{L^{p'}(0,T;X')}^{p'} \leq C(T) \cdot \big(1+\|u_0\|_H^2\big). \qedhere
\end{equation*} 
\end{proof}

\begin{remark}
	We note that we assumed the weak-to-weak continuity of the realization of $u \mapsto \gradH\E(u)$ in order to pass the limit and follow $\gradH\E(u)=\xi$. Alternatively, one could assume use the theory of monotone operators and assume that $\gradH \E$ is of type M, see \cite[Definition, p.38]{Showalter}. In fact, setting $\mathcal{X}=L^p(0,T;X)$ we call $\gradH\E$ of type M if $u_k \rightharpoonup u$ in $\mathcal{X}$, $\gradH\E(u_k) \rightharpoonup \xi$ in $\X'$ and $$\limsup_{k \to \infty} \delta \E(u_k,u_k) \leq \langle\xi,u\rangle_{\X' \times \X},$$ then $\gradH\E(u)=\xi$ in $\X'$. The condition on the limit superior can be followed by the weakly lower semicontinuity of norms.
	
	We note that the energy functional from \cref{Ex:1} fulfills the assumptions from \cref{th:prelim:Existence} for $p=2$ if the function $f$ satisfies some growth bounds, e.g., $f(x) \leq C_1(1+|x|^2)$ and $|f'(x)| \leq C_2 (1+|x|)$ for all $x \in \mathbb{R}$. In fact, the Lebesgue dominated convergence theorem gives as $k \to \infty$
	$$\int_0^T \int_\Omega f(u_k(t,x)) x_j(x) \eta(t) \dx \dt \to \int_0^T \int_\Omega f(u(t,x)) x_j(x) \eta(t) \dx \dt,$$
	which proves $\xi=\gradH$ for this choice of energy.
\end{remark}

In the following corollary, we prove the continuous dependency on the data and the uniqueness of the variational solution to \cref{th:prelim:Existence} under additional assumptions. In particular, we assume higher regularity of the initial data and the $\lambda$-convexity of the energy in order to apply the fractional chain inequality. In case of the Cahn--Hilliard equation and the double-well potential $f(x)=(1-x^2)^2$, it holds $f''(x)=12x^2-4 \geq -4$. Consequently, the function $x \mapsto f(x)+4x^2$ is convex and $f$ is $(-4)$-convex. Consequently, the energy $\E$ is $(-4)$-convex with respect to $L^2(\Omega)$ since
$u \mapsto \E(u) + 4 \int_\Omega u^2 \dx$
is convex.

 \begin{corollary}\label{th:prelim:Existence2}
	Let the assumptions hold from \cref{th:prelim:Existence}.
	Additionally, let $\E$ be $\lambda$-convex for some $\lambda \in \R$ with respect to some Hilbert space $Y \supset X$ and let $u_0 \in X$ such that $\E(\Pi_k u_0) \leq \E(u_0)$. Then it holds that the variational solution $u \in H^{\alpha}(0,T;H)$ is unique, depends continuously on the data, and the energy bound $\E(u(t)) \leq \E(u_0)$ holds for a.e. $t \in (0,T)$.  
\end{corollary}

\begin{proof}
By testing with $\pta u_k$ in the Faedo--Galerkin formulation in \cref{Eq:Faedo}, we have 
\begin{equation*} 
\|\pta u_k\|_H^2 + \delta \E(u_k,\pta u_k) = 0,
\end{equation*}
which yields after the application of the Riemann--Liouville integral operator $\I_{\alpha}$, see \cref{Def:RLintegral},
\begin{equation*} 
	\big(g_\alpha * \|\pta u_k\|_H^2\big)(t) + \big( g_\alpha * \langle \gradH\E(u_k),\pta u_k \rangle_{X' \times X} \big)(t) = 0.
\end{equation*}
Since $\E$ is $\lambda$-convex with respect to $Y$, we can apply the fractional chain inequality, see \cref{Eq:ChainIneq}. Note that any $\lambda$-convex function is $\mu$-convex for all $\mu < \lambda$, and therefore, we assume in this proof w.l.o.g. $\lambda<0$. Applying the fractional chain inequality yields
\begin{equation*} 
	\big(g_{\alpha}*\|\pta u_k\|_H^2\big)(t) + \E(u_k(t)) - \frac{\lambda}{2} \|u_k(t)\|_{Y}^2 \leq \E(\Pi_k u_0) - \frac{\lambda}{2} \|\Pi_k u_0\|_{Y}^2    - \lambda \big( g_{\alpha} *  \langle \pta u_k,u_k\rangle_{X' \times X}\big)(t).
\end{equation*}
The right hand side can be further estimated by
\begin{equation} \label{Eq:Higher}
	\big(g_{\alpha}*\|\pta u_k\|_H^2\big)(t) + \E(u_k(t)) - \frac{\lambda}{2} \|u_k(t)\|_{Y}^2 \leq \E(u_0) - \frac{C\lambda}{2} \|u_0\|_{X}^2    - \lambda \|g_{\alpha}\|_{L^1(0,T)}  \|\pta u_k\|_{L^{p'}(0,T;X')} \|u_k\|_{L^p(0,T;X)},
\end{equation}
which is uniformly bounded due to the auxiliary inequality \cref{Ineq:Aux} and the energy estimate  \cref{Eq:EnergyFinal2} of \cref{th:prelim:Existence}.
 From here, we get a bound of $\E(u_k)$ in $L^\infty(0,T)$. Indeed, it yields $\E(u_k(t))\leq  \E(u_0)$ and consequently, the weakly lower semicontinuity of $\E$ implies $\E(u(t)) \leq \E(u_0)$ for a.e. $t \in (0,T)$. Moreover, by \cref{Eq:Higher} we achieve the uniform bound of $u_k$ in $H^\alpha(0,T;H)$ and $L^\infty(0,T;Y)$, which transfers to the higher regularity of its limit $u$.     \medskip

\noindent \textit{Continuous dependence.} We consider two variational solution pairs $u_1$ and $u_2$ with data $u_{1,0}$ and $u_{2,0}$. We denote their differences by $u=u_1-u_2$ and $u_0=u_{1,0}-u_{2,0}$. Taking their difference yields
\begin{equation*} 
\begin{aligned}
\p_t^\alpha( u,v )_H +  \langle  \gradH\E(u_1)-\gradH\E(u_2),v \rangle_{X' \times X} &= 0,
\end{aligned}
\end{equation*}
for all $v \in V$. Since $\E$ is $\lambda$-convex with respect to $Y$, we have due to the mean value theorem %
\begin{equation*}
\langle \gradH\E(u_1)-\gradH\E(u_2),u \rangle_{X' \times X}  \geq - \lambda \|u\|_Y^2,
\end{equation*}
and therefore, testing with $v=u$  and applying the fractional chain inequality $\frac12\pta  \|u\|_H^2 \leq (\pta u,u)_H$ gives
\begin{equation*} 
\begin{aligned}
\frac12 \pta \|u\|_H^2 - \lambda \|u\|_Y^2   &\leq 0.
\end{aligned}
\end{equation*}
Convolving with $g_\alpha$ and applying the fractional Gr\"onwall--Bellman lemma (setting $w=\|u\|_H^2$, $v=b=0$, $a=\|u_0\|_H^2$ in \cref{Lem:FractionalGronwall1}) finally gives \begin{equation} \label{Eq:ContDep} \|u\|^2_H\leq C(T) \|u_0\|_H^2.\end{equation}

\noindent \textit{Uniqueness.} The proof follows analogously to the procedure of continuous dependency but with the same initial conditions $u_{1,0}$ and $u_{2,0}$. Hence, from \cref{Eq:ContDep} it holds $\|u\|_H^2=0$ and hence $u_1=u_2$ in $H$ for a.e. $t \in (0,T)$.
\end{proof}

\begin{remark}
	 We note that we were only able to prove the  $\E(u(t)) \leq \E(u(0))$, $t \geq 0$, which does not imply that the energy is monotonically decaying over time. In the integer-order $\alpha=1$ setting we immediately achieve $\E(u(t))\leq \E(u(s))$, $t \leq s$, by the same lines. The effect at $t=0$ is of most importance in the study of time-fractional differential equations. This can be also seen from the inequality \cite[Lemma 3.1]{mustapha2014well}
	 $$\int_0^t (\I_{\alpha} v,v)_H \ds \geq \cos\left(\frac{\alpha \pi}{2}\right) \|\I_{\alpha/2} v\|_{L^2(0,t;H)}^2 \qquad \forall v \in C([0,T];H),$$
	 where the integral starts from $t=0$. Therefore, testing with
$v=\pt u$ in the variational form \cref{Eq:GradientVariational} yields by the classical chain rule and integration over the interval $(s,t)$
	\begin{equation*} \label{Eq:EstimateInitial}
		 \E(u(t)) = \E(u(s)) - \int_s^t (\pta  u ,\pt u)_H \, \dd \tau.
	\end{equation*}
	 For $s=0$ we can apply the inequality from above which yields again a bound by the initial state:  
	\begin{equation*} \label{Eq:EnergyCos}
		 \E(u(t)) \leq \E(u(0))- \cos\left(\frac{(1-\alpha) \pi}{2}\right) \|\pt^{(1+\alpha)/2}  u\|_{L^2(0,t;H)}^2 \leq \E(u(0)).
	\end{equation*}
\end{remark}

%% file: 2_analysis.tex
\section{Augmented Gradient Flow and Energy Dissipation}
\label{Sec:Main}

In this section, we give one of the main results of this article. We introduce a new energy functional and prove the equivalence of the fractional gradient flow to an integer-order gradient flow corresponding to the new energy functional.

\subsection{Motivation: Extension of the dimension}
Let $H$ be a Hilbert space over the bounded domain $\Omega \subset \R^d$, $d \geq 1$, and let $u:(0,T) \to H$ be a Bochner measurable function with $T>0$ being a finite time horizon. 
We introduce a function $\tu$ which acts on an extended domain, and we define  \begin{equation*} \label{Eq:DefTU}
	 \tu : (0,T) \times (0,1) \times \Omega \to \R.
\end{equation*}
As above, we want to interpret it as a function mapping to a Hilbert space and therefore, we define
\begin{equation*}
	\tu : (0,T) \to L^2(0,1;H), 
\end{equation*}
such that  $\tu(t)(\theta,x) = \tu(t,\theta,x)$ for almost every $t \in (0,T)$, $\theta \in (0,1)$, and $x \in \Omega$. We refer to \cref{Fig:Sketch} for a sketch of the domains of the respective functions $u$ and $\tu$.

Further, we assume that the energy $\E:X \to \R$ is Fr\'echet differentiable and satisfies the assumptions of \cref{th:prelim:Existence}.
We consider the following two gradient flow problems: the original time-fractional gradient flow in the Hilbert space~$H$, see \cref{Eq:GradientVariational},
\begin{equation}\label{eq:frational_GF2}
	\pta u = - \gradH \E(u), \qquad u(0)=u_0.
\end{equation}
and a higher-dimensional integer-order gradient flow in the augmented Hilbert space~$\tH$,
\begin{equation}\label{eq:integer_GF}
\pt \tu = - \gradHt\tE(\tu), \qquad \tu(0)=0.
\end{equation}
Under suitable assumptions of $\tE$, we have that $\tu \mapsto \tE(\tu)$ is non-increasing.
We call \cref{eq:integer_GF} the augmented gradient flow in the sense that we want to find an appropriate form of the associated energy functional~$\tE$ and the Hilbert space~$\tH$, depending on~$\alpha$, which provide an equivalence of the variational solutions $u$ and $\tu$.

\begin{figure}[!htb]
	\centering
	\includegraphics[width=.7\textwidth]{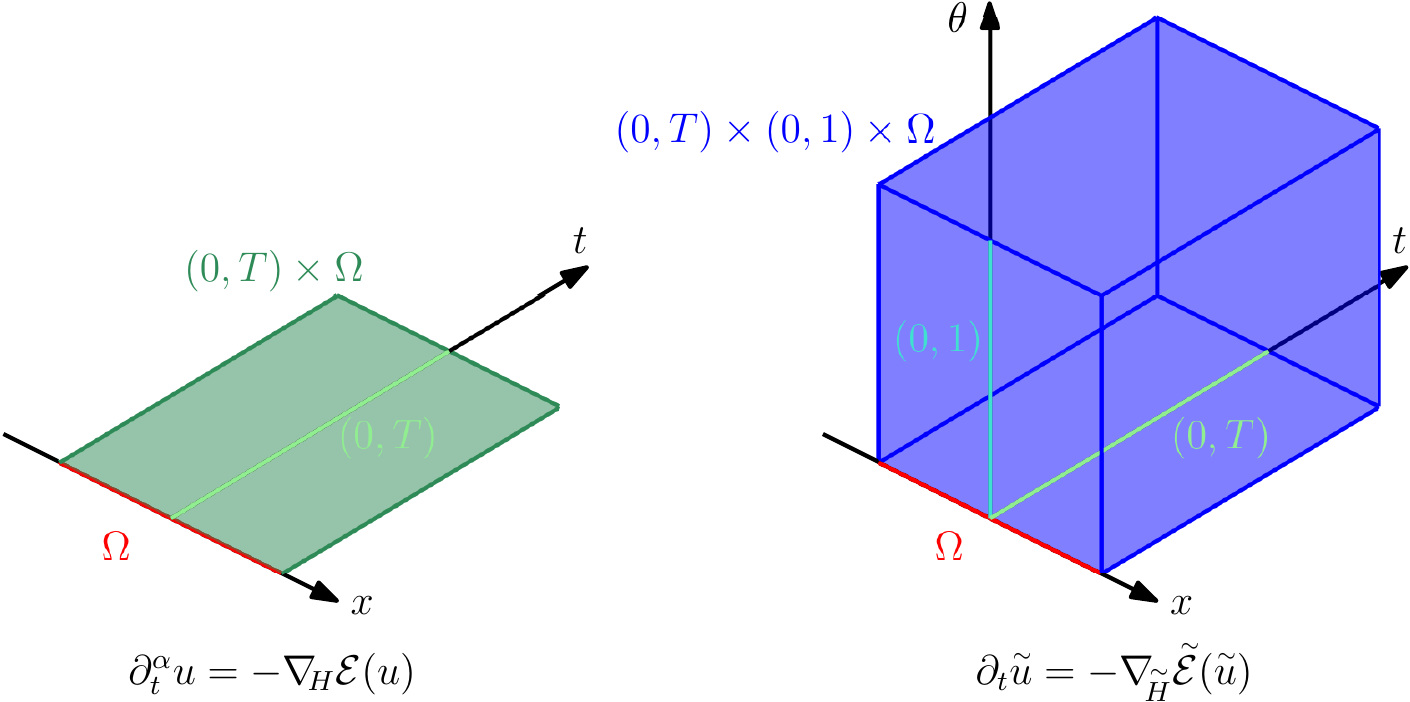}
	\caption{Left: Time-space domain $(0,T) \times \Omega$ of the function $u$. Right: Extended domain of $\widetilde u$. \label{Fig:Sketch}}
\end{figure}

\begin{remark} The idea of the dimension extension is reminiscent of the Caffarelli--Silvestre method applied to the fractional Laplacian, see \cite{caffarelli2007extension}. Indeed, let $(-\Delta)^\alpha$, $\alpha>0$, be the nonlocal fractional Laplacian acting on $\R^d$. Define a function $\tu:\R^d \times (0,\infty) \to \R$ with $\tu(x,0)=u(x)$ in $\R^d$ such that it is the solution to the local degenerate differential equation 
\begin{equation*}
		\nabla_{\!(x,\lambda)} \cdot (\lam^{1-2\alpha}\nabla_{\!(x,\lambda)} \tu(x,\lam))=0,
\end{equation*}
	in $\R^d \times (0,\infty)$.
	Then, one has the relation
\begin{equation*}
		(-\Delta)^\alpha u(x)=-C\lim_{\lam \to 0^+} (\lam^{1-2\alpha} \p_\lam \tu(x,\lam)).
\end{equation*}
	Thus, one recovers a local PDE from a nonlocal one. We refer to \cite{bonito2018numerical,nochetto2015pde,banjai2019tensor} for numerical methods which exploit this equivalence.
\end{remark}

\subsection{Equivalency between fractional and integer-order gradient flow}

Let us introduce the following functions of $\theta\in(0,1)$:
\begin{align}\label{eq:def_kernels}
\begin{aligned}
w_{\alpha}(\theta) &:= g_{1-\alpha}(\theta)\,g_{\alpha}(1-\theta),\\
w_{\alpha,0}(\theta) &:= w_{\alpha}(\theta)\cdot(1-\theta),\\
w_{\alpha,1}(\theta) &:= w_{\alpha}(\theta)\cdot\theta,
\end{aligned}
\end{align}
which are plotted in~\cref{fig:kernels} for different values of~$\alpha$. Note that the functions in~\eqref{eq:def_kernels} are integrable for all $\alpha\in (0,1)$.
Indeed, we have
\begin{equation} \begin{alignedat}{4}
	\int_{0}^{1}w_{\alpha}(\theta)\dth &= \;\;\frac{B(1-\alpha,\alpha)}{\Gamma(1-\alpha)\Gamma(\alpha)} &&=1,
	\\
	\int_{0}^{1}w_{\alpha,0}(\theta)\dth &= \frac{B(1-\alpha,1+\alpha)}{\Gamma(1-\alpha)\Gamma(\alpha)}
	&&= \alpha,
\\
	\int_{0}^{1}w_{\alpha,1}(\theta)\dth &= \;\;\frac{B(2-\alpha,\alpha)}{\Gamma(1-\alpha)\Gamma(\alpha)}
	&&= 1-\alpha,
\end{alignedat}	 \label{eq:int_w}  \end{equation}
where $(x,y) \mapsto B(x,y)$ denotes the Euler beta function \cite{abramowitz1988handbook}.
Therefore, we consider the following weighted Lebesgue spaces:
\begin{equation*}
	\begin{aligned}
		L^2_{\alpha}(0,1) 	&:=L^2(0,1; w_{\alpha}), 	\\
		L^2_{\alpha,0}(0,1) &:=L^2(0,1; w_{\alpha,0}),	\\
		L^2_{\alpha,1}(0,1) &:=L^2(0,1; w_{\alpha,1}).
	\end{aligned}
\end{equation*} 
In particular, given a Hilbert space~$H$ and a weight~$w$, the norm in the $w$-weighted Bochner space $L^2(0,1;H;w)$ is induced by the inner product
\begin{equation*}\label{key}
	(\tu,\tv)_{L^2(0,1;H;w)} = \int_{0}^{1}(\tu(\theta),\tv(\theta))_{H}\,w(\theta)\dth,
	\qquad\forall\tu,\tv\in L^2(0,1;H;w).
\end{equation*}

\begin{figure}[!htb]
	\centering
	\includegraphics[width=.9\textwidth]{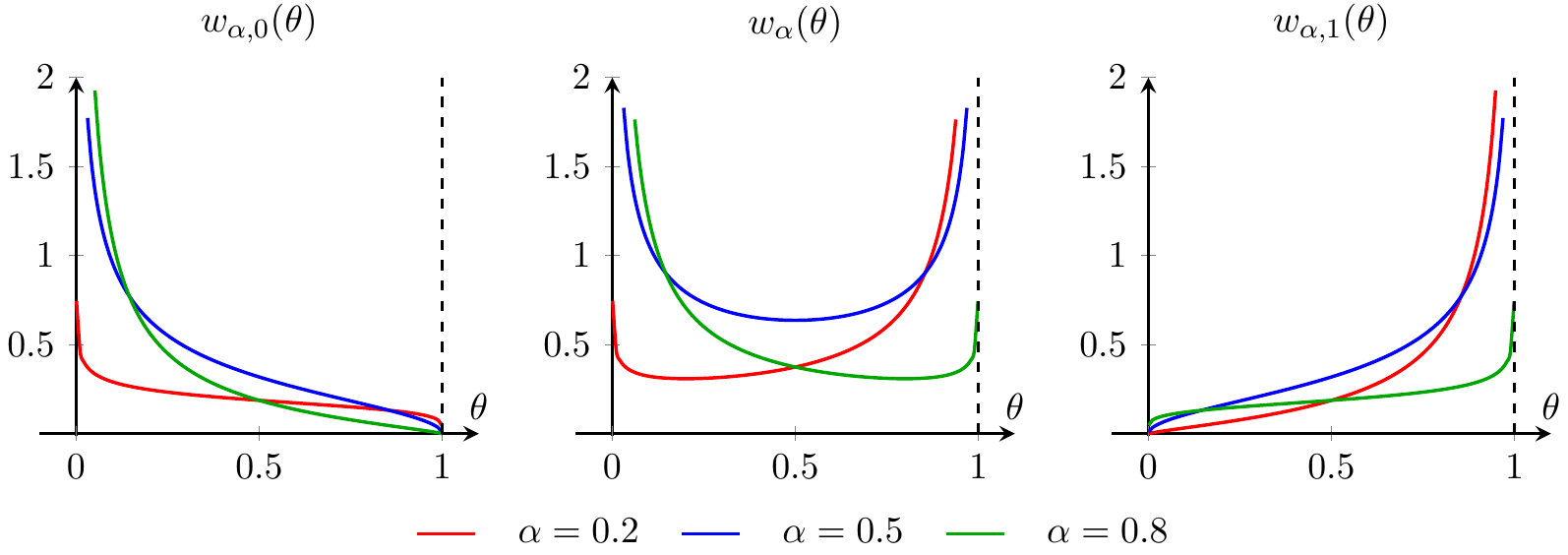}
	\caption{Depiction of the kernels $w_{\alpha,0}$ (left), $w_{\alpha}$ (middle), and $w_{\alpha,1}$ (right) for $\alpha \in \{ 0.2, 0.5, 0.8\}$.} %
\label{fig:kernels}
\end{figure}

	\begin{lemma}
	For $\alpha\in(0,1)$, the following continuous embeddings hold:
	\begin{equation*}\label{key}
	\begin{aligned}
	L^2_{\alpha}(0,1) &\con L^2_{\alpha,i}(0,1),
	\quad i\in\{0,1\},\\
	L^2_{\alpha}(0,1) &\con L^2(0,1).
	\end{aligned}	 
	\end{equation*}
\end{lemma}

\begin{proof}
	The first embedding directly follows from the expression
	\begin{equation*}\label{key}
		\norm{\tv}_{L^2_{\alpha}(0,1)}^2 = \norm{\tv}_{L^2_{\alpha,0}(0,1)}^2 + \norm{\tv}_{L^2_{\alpha,1}(0,1)}^2 \ge \norm{\tv}_{L^2_{\alpha,i}(0,1)}^2,\quad i\in\{0,1\}.
	\end{equation*}
	Next, we investigate the maximum of $w_{\alpha}(\theta)^{-1}$.
	To this end, we compute its derivative:
	\begin{equation*}\label{key}
	\begin{aligned}	
	\frac{\p_\theta \left(w_{\alpha}(\theta)^{-1}\right)}{\Gamma(\alpha)\,\Gamma(1-\alpha)} 
	=
	 \p_\theta \left( \theta^{\alpha}(1-\theta)^{1-\alpha} \right)
	&= \bigl[ \alpha \theta^{\alpha-1}(1-\theta)^{1-\alpha} - (1-\alpha) \theta^{\alpha}(1-\theta)^{-\alpha}\bigr]
	\\
	&= \theta^{\alpha-1}(1-\theta)^{-\alpha} \bigl[ \alpha (1-\theta) - (1-\alpha) \theta\bigr],
	\end{aligned}
	\end{equation*}
	which vanishes at $\theta=\alpha$.
	Given that $w_{\alpha}(\theta)$ has only one extremum which is a minimum (see~\Cref{fig:kernels}), we conclude that $\max_{\theta\in[0,1]}w_{\alpha}(\theta)^{-1} = w_{\alpha}(\alpha)^{-1} \in (0,\infty)$ for $\alpha\in(0,1)$.
	Then, by the H\"older inequality it yields the missing result
	\begin{equation*}\label{key}
	\norm{\tv}_{L^2(0,1)} 
	\le w_{\alpha}(\alpha)^{-\frac{1}{2}}\,\norm{\tv}_{L^2_{\alpha}(0,1)}. \qedhere
	\end{equation*}
\end{proof}

Let us define the following two functions: \begin{equation*}
	\begin{aligned}
		c_0(\theta) :=\frac{1}{1-\theta}, \qquad
		c_1(\theta) :=\frac{\theta}{1-\theta}.
	\end{aligned}
\end{equation*}
Note that it holds that $w_{\alpha}=c_0 w_{\alpha,0}$ and $w_{\alpha,1}=c_1 w_{\alpha,0}$.
Correspondingly, for any function $\tv \in L^2_{\alpha}(0,1;H)$, $\alpha \in (0,1)$, we define an integral operator $\mathcal{C}:L^2_{\alpha}(0,1;H)\to H$ and a functional~$\H:L^2_{\alpha,1}(0,1;H)\to\R_{\ge0}$ as follows
\begin{equation} \begin{alignedat}{4}
\mathcal{C}\tv &:= u_0 + (1,\tv)_{L^2_{\alpha}(0,1)}
&=
u_0 + (c_0,\tv)_{L^2_{\alpha,0}(0,1)}
&=
u_0 + \int_{0}^{1} g_{1-\alpha}(\theta) g_\alpha(1-\theta) \tv(\theta)  \dth,
\\
\H(\tv) 
&:= \frac12\,\norm{\tv}_{L^2_{\alpha,1}(0,1;H)}^2
&= \frac12(c_1\tv,\tv)_{L^2_{\alpha,0}(0,1;H)}
&=
\frac12 \int_{0}^{1} \theta g_{1-\alpha}(\theta) g_\alpha(1-\theta)  \norm{\tv\left(\theta\right)}_H^2 \dth.
\end{alignedat} \label{eq:def_HC} \end{equation}
We call $\H$ the history energy.
The following lemma proposes an integral representation in $L_{\alpha}^2(0,1)$ of the fractional kernel~$g_\alpha$.
\begin{lemma}\label{th:kernel_rep}
	It holds that
	$(1, c_0 e^{-c_1 t})_{L_{\alpha}^2(0,1)} = g_\alpha(t)$ for all $t\in[0,T]$, $\alpha \in (0,1)$.
\end{lemma}

\begin{proof}
	First, let us remark that it holds $\Gamma(1-\alpha)\Gamma(\alpha)w_{\alpha}(\theta)=c_0(\theta)\,\abs{c_1(\theta)}^{-\alpha}$ and $\abs{c_0(\theta)}^2 = c_1^\prime(\theta)$.
	Then, with the change of integration variable $\lambda:= c_1(\theta)$, we can write
	\begin{equation}\label{eq:kernel_rep:proof}
		(1, c_0 e^{-c_1 t})_{L_{\alpha}^2(0,1)}
		=
		\int_{0}^{1} w_{\alpha}(\theta)\, e^{-c_1(\theta)t} \abs{c_0(\theta)}^2\dth
		=
		\frac{1}{\Gamma(1-\alpha)\Gamma(\alpha)}\int_{0}^{\infty} \lambda^{-\alpha}\, e^{-\lambda t} \dl.
	\end{equation}
	By definition of the Euler's gamma function \cite{abramowitz1988handbook}, it holds $\Gamma(1-\alpha)t^{\alpha-1}=\int_{0}^{\infty} \lambda^{-\alpha}\, e^{-\lambda t}\dl$.
	Using this in~\eqref{eq:kernel_rep:proof}, we end up with $(1, c_0 e^{-c_1 t})_{L_{\alpha}^2(0,1)}
	= g_\alpha(t)$ for all $t \in [0,T]$.
\end{proof}

\begin{lemma}\label{th:Cv_estimate} \label{th:Hv_est}
	The operators $\mathcal{C}$ and $\H$, see \cref{eq:def_HC}, are bounded.
	In particular, for all $\tv \in L^2_{\alpha}(0,1;X)$, $\alpha \in (0,1)$, it holds that
\begin{equation*} 
	\norm{\mathcal{C}\tv-u_0}_{X}
	\le
	\norm{\tv}_{L^2_{\alpha}(0,1;X)}.
\end{equation*}
And for all $\tv \in L^\infty(0,1;L^{p}(0,T;H))$ and $p>1$, it holds the upper bound
\begin{equation}\label{Eq:BoundH}
\norm{\H(\tv)}_{L^{p/2}(0,T)} \le \frac{1-\alpha}{2}\cdot\norm{\tv}_{L^\infty(0,1;L^{p}(0,T;H))}^2.
\end{equation}
If $\alpha$ goes to $1$, then it follows that $\H(\tv(t))$ vanishes for a.a. $t \in (0,T)$ for all $\tv\in L^\infty(0,1;L^{p}(0,T;H))$.
\end{lemma}

\begin{proof}
	Note that by~\eqref{eq:int_w}, we have $\norm{w_{\alpha}}_{L^1(0,1)} = (1,w_{\alpha})_{L^2(0,1)} = 1$.
	Then, the following estimate holds
	\begin{equation*}\label{key}
	\norm{\mathcal{C}\tv-u_0}_{X}
	\le
	\int_{0}^{1}w_{\alpha}(\theta) \,\norm{\tv(\theta)}_{X}\dth
	\le
	\norm{w_{\alpha}}_{L^1(0,1)}^{\frac{1}{2}} \cdot \norm{\tv}_{L^2_{\alpha}(0,1;X)}.
	\end{equation*}
	Besides, by~\eqref{eq:int_w} we have $\norm{w_{\alpha,1}}_{L^1(0,1)} = (1,w_{\alpha,1})_{L^2(0,1)} = 1-\alpha$. 
	Then, using this and the triangle inequality, we can write
\begin{align*}
	\norm{\H(\tv)}_{L^{p/2}(0,T)} &\leq \int_0^1 \Big\| \frac12 w_{\alpha,1}(\theta) \|\tv(\theta)\|_H^2 \Big\|_{L^{p/2}(0,T)} \dth \\ 
	&= 
	\frac12 \int_{0}^{1}w_{\alpha,1}(\theta) \,\norm{\tv\left(\theta\right)}_{L^{p}(0,T;H)}^2\dth \\
	&\le \frac{1-\alpha}{2}\cdot\norm{\tv}_{L^\infty(0,1;L^{p}(0,T;H))}^2.  \qedhere
	\end{align*}	
\end{proof}

Now we are ready to prove our main result stating the equivalence of the time-fractional gradient flow \eqref{eq:frational_GF2}  and its integer-order counterpart \eqref{eq:integer_GF}.

\begin{theorem}\label{th:energy}
	Let the assumptions from \cref{th:prelim:Existence} hold. Further, let $\tX=L^2_{\alpha}(0,1;X)$ and $\tH=L^2_{\alpha,0}(0,1;H)$ with $\alpha \in (0,1)$.  We assume from now on that the energy functional~$\tE$ is of the form
	\begin{equation}\label{eq:tE(E)}
	\tE
	=
	\E\circ \mathcal{C}
	+ \H,
	\end{equation}
	with $\mathcal{C}$ and $\H$ defined in~\eqref{eq:def_HC}.
	Then, for any solution $u$ to the variational form of~\eqref{eq:frational_GF2} there exists a solution  $\tu$ to the variational form of~\eqref{eq:integer_GF}, and vice-versa, such that the following equivalence of solutions holds: 
	\begin{align}
		u(t) &=\mathcal{C}\tu(t) = u_0 + \int_{0}^{1}g_{1-\alpha}(\theta)\,g_{\alpha}(1-\theta) \tu(t,\theta)\dth,
		\label{eq:u(tu)}
		\\
		\tu(t,\theta) &=\G{u}(t,\theta) :=  -c_0(\theta)\int_{0}^{t}e^{-c_1(\theta) (t-s)}\gradH \E(u(s))\ds.
		\label{eq:tu(u)}
	\end{align}
	Moreover, we have the regularity result
	\begin{equation*}\label{eq:spaces}
	\begin{aligned}
	u &\in L^p(0,T;X) \cap L^\infty(0,T;H) \cap H^\alpha(0,T;H),\\
	\tu &\in L^\infty(0,T;L^2_{\alpha}(0,1;H)) \cap H^1(0,T;L^2_{\alpha,0}(0,1;H))
\cap L^{\infty}(0,1; L^{p'}(0,T;H)).	
	\end{aligned}		
	\end{equation*}
\end{theorem}

\begin{proof} We separate the proof into three steps. First, we assume a variational solution $\tu$ of \cref{eq:integer_GF} and show that $\mathcal{C} \tu$ is a solution to \cref{eq:frational_GF2}. Second, we assume a solution $u$ of \cref{eq:frational_GF2} and prove thar $\G{u}$ is a solution to \cref{eq:integer_GF}. Lastly, we prove the stated regularity of the solutions.
	\medskip
	
\noindent \textit{Augmented to fractional}.	Let $\tu$ be a solution to~\eqref{eq:integer_GF}.
	The G\^ateaux derivative of the energy functional~$\tE$ of the form~\eqref{eq:tE(E)} reads
	\begin{equation*}
			\delta \tE(\tu,\tv)
			= 
			(c_0,\delta\E(\mathcal{C}\tu,\tv))_{L^2_{\alpha,0}(0,1)} +  ( c_1 \tu,  \tv)_{\tH}
			=
			(c_0 \gradH\E(\mathcal{C}\tu) + c_1 \tu,  \tv)_{\tH}.
	\end{equation*}
	for all $\tv \in \tX$. 
Hence, the $\tH$-gradient of $\tE$ is given by $\gradHt\tE(\tu) = c_1 \tu + c_0 \gradH\E(\mathcal{C}\tu)$, and the variational form of the associated gradient flow in $\tH$ can be thus written as
	\begin{equation} \label{Eq:Diffeq}
		\langle  \p_t \tu(t)  
		+
		c_1 \tu(t)
		+ 
		c_0 \gradH\E(\mathcal{C}\tu(t)) 
		, \tv \rangle_{\tX'\times \tX}
		= 0,
		\qquad \forall \tv \in \tX.
	\end{equation}
The solution to this ODE with zero initial condition $\tu(0)=0$ can be formally written in the form
	\begin{equation}\label{eq:solutionode}
		\tu(t,\theta) = -c_0(\theta)\int_{0}^{t}e^{-c_1(\theta)\cdot (t-s)}\gradH\E(\mathcal{C}\tu(s))\ds,
	\end{equation}
	in $X'$ for all $t\in(0,T)$ and $\theta \in [0,1]$.
	Indeed, we have by the Leibniz integral formula
\begin{align*}
		\pt \tu(t,\theta) 
		&= -c_0(\theta) \gradH\E(\mathcal{C}\tu(t)) + c_0(\theta)c_1(\theta) \int_0^t  e^{-c_1(\theta) (t-s)} \gradH\E(\mathcal{C}\tu(s)) \ds 
		\\
		&= -c_0(\theta) \gradH\E(\mathcal{C}\tu(t)) - c_1(\theta)\tu(t).
\end{align*}
Then, applying the operator $\mathcal{C}$ to~\eqref{eq:solutionode}, we can write
\begin{equation*} \label{eq:uafterode}
		\mathcal{C}\tu(t)=
		u_0-\int_{0}^{t}(1, c_0 e^{-c_1\cdot(t-s)})_{L_{\alpha}^2(0,1)} \cdot \gradH\E(\mathcal{C}\tu(s))\ds.
\end{equation*}
By \cref{th:kernel_rep}, the kernel reads $(1, c_0 e^{-c_1\cdot(t-s)})_{L_{\alpha}^2(0,1)} = g_\alpha(t-s)$, and thus it holds
\begin{equation*}
\mathcal{C}\tu = u_0- \I_{\alpha} \gradH\E(\mathcal{C}\tu).
\end{equation*}
Given~\cref{Eq:InverseConvolution} and $\mathcal{C}\tu(0)=u_0$, it can be written as
\begin{equation*}
\I_{\alpha}\left[\pta\mathcal{C}\tu +  \gradH\E(\mathcal{C}\tu)\right] = 0.
\end{equation*}
Hence we conclude that $\mathcal{C}\tu$ is a solution to~\eqref{eq:frational_GF2}.
	\medskip

\noindent \textit{Fractional to augmented}.	 Now, let us assume that $u$ is a solution to~\eqref{eq:frational_GF2}.
Then, it satisfies $u = u_0- \I_{\alpha} \gradH\E(u)$.
Besides, for $\tu=\G{u}$ given by~\eqref{eq:tu(u)}, applying the operator~$\mathcal{C}$ and \cref{th:kernel_rep}, we obtain
\begin{equation*}\label{key}
	\mathcal{C}\tu = u_0- \I_{\alpha} \gradH\E(u).
\end{equation*}
Hence it follows that $\mathcal{C}\tu=u$, and therefore $\tu$ writes in the form~\cref{eq:solutionode}, which is a solution to~\cref{Eq:Diffeq}.
Thus, we conclude that $\G{u}$ is a solution to~\eqref{eq:integer_GF}.
\medskip
	
\noindent \textit{Regularity.} Finally, let us comment on the solutions regularity. 
The existence of a solution to \eqref{eq:frational_GF2} with regularity \begin{equation*} u \in L^p(0,T;X) \cap L^\infty(0,T;H) \cap W^{\alpha,p'}(0,T;H),\end{equation*} directly follows from \cref{th:prelim:Existence}. Similarly to \cref{th:prelim:Existence}, we proceed in a discrete setting in order to derive suitable energy estimates. Afterwards, one passes to the limit by the same type of arguments. We skip the details and directly state the estimates for~$\tu$.
Let us test equation~\eqref{Eq:Diffeq} with $\tv = \p_t\tu + \tu \in \tH$  to obtain
\begin{equation*}
\norm{\p_t\tu(t)}_{\tH}^2
+ 
\frac12\ddt\norm{\tu(t)}_{\tH}^2
+
\ddt \H(\tu(t))
+
\ddt \E(u(t))
+
2\H(\tu(t))
+
(\gradH\E(u(t)),u(t)-u_0)_H
= 0.
\end{equation*}
Note that it holds 
\begin{equation*} 
\norm{\tv}_{L^2_{\alpha,0}(0,T;H)}^2 + \norm{\tv}_{L^2_{\alpha,1}(0,T;H)}^2 = \norm{\tv}_{L^2_{\alpha}(0,T;H)}^2,
\qquad
\H(\tv)=\frac12\norm{\tv}_{L^2_{\alpha,1}(0,T;H)}^2.
\end{equation*}
Then, owing to the zero initial condition for $\tu$, integration in time from zero to $t\le T$ results in
\begin{equation*}
\begin{split}
\norm{\p_t\tu}_{L^2(0,t;\tH)}^2 
&+ 
\frac12\norm{\tu(t)}_{L^2_{\alpha}(0,1;H)}^2
\\ 
&\le 
\E(u_0)-\E(u(t)) + C_0\,\left(t + \norm{u}_{L^p(0,t;X)}^{p}\right)\,\norm{u_0}_X + C_2 \norm{u}_{L^2(0,t;H)}^2,
\end{split}
\end{equation*}
where we also used \eqref{eq:continuity}-\eqref{eq:coercivity}.
Hence, we end up with $\tu \in L^\infty(0,T;L^2_{\alpha}(0,1;H)) \cap H^1(0,T;L^2_{\alpha,0}(0,1;H))$. Lastly, we show that $\tu \in L^{\infty}(0,1; L^{p'}(0,T;H))$. Indeed, we have by Young's convolution inequality \cite[Lemma~A.1]{kubica}
	\begin{align*} \sup_{\theta \in (0,1)} \int_0^T   \|\tu(t,\theta)\|_H^{p'}\dt %
		 &= \sup_{\theta \in (0,1)} \left( |c_0(\theta)|^{p'} \int_0^T  \big(e^{-c_1(\theta)(\cdot)}* \|\gradH\E(u(s))\|_H\big)(t)^{p' } \dt \right)
		\\ &= \sup_{\theta \in (0,1)} |c_0(\theta)|^{p'} \|e^{-c_1(\theta)(\cdot)}* \|\gradH\E(u(s))\|_H \|_{L^{p'}(0,T)}^{p'} 
		\\ &\leq \sup_{\theta \in (0,1)} \left( \int_0^T |c_0(\theta)| e^{-c_1(\theta)t} \dt \right)^{p'} \cdot \|\gradH\E(u(s)) \|_{L^{p'}(0,T;H)}^{p'}  
		\\ &\leq C(T,u_0) \sup_{\theta \in (0,1)} \left( \int_0^T |c_0(\theta)| e^{-c_1(\theta)t} \dt \right)^{p'},
	\end{align*}
	and the integral on the right hand side is bounded owing to %
	\begin{equation*}
		c_0(\theta) \int_0^T e^{-c_1(\theta)t} \dt =(1+c_1(\theta))\, \frac{1-e^{-c_1(\theta)T}}{c_1(\theta)} =  T\,\frac{1-e^{-c_1(\theta)T}}{c_1(\theta)T} + 1-e^{-c_1(\theta)T}\leq T+1. \qedhere
	\end{equation*}

\end{proof}

\begin{remark}
	Remark that the augmented gradient flow system~\eqref{Eq:Diffeq} writes
	\begin{equation}\label{eq:rmk:ode}
	(1-\theta)\,\p_t \tu(t,\theta)  
	+
	\theta\,\tu(t,\theta)
	+ 
	\gradH\E(u(t))
	= 0
	\end{equation}
	in $L^2_{\alpha}(0,1;H)$.
	That is, for $\theta\in(0,1)$, $\tu(t,\theta)$ presents a smooth transition from $\tu(t,0) = -\int_{0}^{t}\gradH\E(u(s))\d s$ to $\tu(t,1) = -\gradH\E(u(t))$.
	Thus, the fractional gradient flow solution \cref{eq:u(tu)}
	corresponds to a weighted average of the transient solution $\tu(t,\theta)$.
\end{remark}

\begin{remark}\label{rmk:SteadyState}
	Note that in case of a steady state $\pt\tu=0$, the equation~\cref{eq:rmk:ode} results in $\gradH\E(u)= 0$ when $\theta=0$ owing to the regularity specified in \cref{th:energy}, and in $\tu = -\theta^{-1}\gradH\E(u)$ when $\theta\ne 0$.
	Thus, the modes $\tu(t,\theta)$, $\theta\ne 0$, vanish at the steady state, and therefore $\H(\tu)$ also vanishes.
\end{remark}

\subsection{Augmented energy and memory contribution}

The solution to the time-fractional gradient flow problem~\eqref{eq:frational_GF2} can be represented in the form $u=\mathcal{C}\tu$ as a linear combination of different modes~$\tu(\cdot,\theta)$, with index $\theta$, solutions to the classical gradient flow~\eqref{eq:integer_GF}.
Moreover, we can formally define 
the augmented total energy
\begin{equation}\label{eq:gen_energy}
	\E^\text{aug}(u)
	=
	\E(u)
	+ \H(\G{u})
	= 	\E(\mathcal{C}\tu)
	+ \H(\tu) = 
	\tE(\tu),
\end{equation}
where the history part $\H(\tu)$ corresponds to the memory contribution.
In the following lemma, we show that $\H(\tu)$ is bounded and that the memory effects vanish when $\alpha=1$ (i.e., for the classical gradient flow).

\begin{corollary} \label{cor:vanish}
	The memory energy contribution $\H(\tu(t))$ is bounded for a.a. $t\in[0,T]$, $\alpha \in (0,1)$, and vanishes when $\alpha$ goes to $1$. 
\end{corollary}

\begin{proof}
	By definition, it holds for the history part of the energy that
	\begin{equation*} \H(\tu(t))=\frac{1}{2}\norm{\tu(t)}_{L^2_{\alpha,1}(0,1:H)}^2 = \frac{1}{2}\norm{\tu(t)}_{L^2_{\alpha}(0,1;H)}^2 - \frac{1}{2}\norm{\tu(t)}_{L^2_{\alpha,0}(0,1;H)}^2,\end{equation*} 
	for all $t\ge 0$.
	According to~\cref{th:energy}, we have $\tu \in L^\infty(0,T;L^2_{\alpha}(0,1;H))$, hence it follows $\H(\tu)\in L^\infty(0,T)$.	
	Besides, it directly follows from~\cref{Eq:BoundH} that $\H(\tu(t))=0$ for a.e. $t\in[0,T]$ if $\alpha=1$.	
\end{proof}

Eventually, as a direct consequence of the equivalence to an integer-order gradient flow, we can prove the dissipation of the augmented energy.

\begin{corollary}
	The augmented energy~\eqref{eq:gen_energy} is monotonically decreasing in time.
\end{corollary}

\begin{proof}
	According to~\cref{th:energy}, the time-fractional gradient flow is equivalent to the integer-order gradient flow 
	and thus, $\ddt \tE(\tu) = -\norm{\drv{t}{\tu}}^2_\tH \leq 0$
	by the chain rule.
\end{proof}

\begin{contexample} 
We return to the example of \cref{Ex:1} and investigate the representation of the history energy $\mathcal{H}$ for these specific gradient flows. We have the Allen--Cahn and Cahn--Hilliard equations
\begin{equation*}
		\pta u = - \gradH\E(u) = \begin{cases}
	-\mu, &H=L^2(\Omega), \\   
	\Delta \mu, &H=\dot H^{-1}(\Omega),\end{cases}
\end{equation*}
	where we defined the chemical potential $\mu=f'(u)-\Delta u$.
	Then, we can compute $\tu$ according to \cref{eq:solutionode} and calculate the augmented energy $\tE$ as given in \cref{eq:gen_energy}.
	In case of our examples, it yields for $\tu$:
	\begin{equation*}
			\tu(t,\theta) = c_0(\theta)\int_{0}^{t}e^{-c_1(\theta) \cdot (t-s)}\begin{lrdcases}
	 -\mu \\   
	 \Delta \mu
		\end{lrdcases}\ds, 
	\end{equation*}
	and for the augmented energy:
\begin{equation*}
\begin{aligned}
	\E^\textup{aug}(u) 
	= \E(u) + \H(\tu)
	&=
	\E(u)  + \frac12 \int_0^1 \|\tu(t,\theta)\|_H^2 \, w_{\alpha,1}(\theta) \, \dd \theta
	\\
	&= \E(u) + \frac12 \int_0^1 \int_\Omega \begin{lrdcases} 
|\tu(t,\theta,x)|^2  \\    |\nabla (-\Delta)^{-1} \tu(t,\theta,x)|^2 \end{lrdcases} \dx \, w_{\alpha,1}(\theta)\, \dd \theta,
\end{aligned}
\end{equation*}
which can be written in terms of the chemical potential $\mu$ as
\begin{equation*}
		\begin{aligned}		\E^\textup{aug}(u) 
&=
\E(u) +\frac12 \int_0^1 \int_\Omega c_0(\theta)^2 \left| \int_0^t e^{-c_1(\theta) \cdot (t-s)} \begin{lrdcases} 
	\mu  \\    \nabla \mu \end{lrdcases} \ds \right|^2 \dx \, w_{\alpha,1}(\theta)\, \dd \theta.\end{aligned}
\end{equation*}

\end{contexample}

%% file: 3_numerics.tex
\section{Numerical algorithm} \label{Sec:Num}

Numerical methods for approximating the solution of a time-fractional differential equation are a fast developing research field.
For a detailed overview of the existing methods, we refer to~\cite{diethelm2020good,diethelm2008investigation,baleanu2012fractional,owolabi2019numerical}.
Among classical methods related to the quadrature of the fractional integral~\cite{lubich1983stability,lubich1986discretized,lubich1988convolution,diethelm2002predictor,diethelm2004detailed,zayernouri2016fractional,zhou2019fast}, there is a class of so-called kernel compression methods based on the approximation of the spectrum of the fractional kernel~\cite{Li2010,mclean2006time,jiang2017fast,zeng2018stable,baffet2019gauss,baffet2017kernel,banjai2019efficient}.
The general idea of such methods is to approximate the fractional kernel with a sum of exponentials, which leads to a system of local ODEs similar to~\cref{eq:rmk:ode}.
In our numerical experiments, we make use of the scheme proposed in~\cite{khristenko2021solving}, which is based on the rational approximation of the Laplace spectrum of the fractional kernel.

	A unified algorithm for the numerical approximation of $\mathcal{C} \tu$, see~\cref{eq:u(tu)}, can be seen as a quadrature rule for the $w_{\alpha}$-weighted integral
	\begin{equation*}\label{key}
	u(t)-u_0 = \int_{0}^{1}w_{\alpha}(\theta)\,\tu(\theta)\dth \approx \sum_{k=1}^{m}w_k \tu_k(t),
	\end{equation*}
	with $m\in \mathbb{N}$ quadrature nodes~$\theta_k$, the quadrature weights~$w_k$ and the modes~$\tu_k:=\tu(t,\theta_k)$ defined by the integer-order differential system:
	\begin{equation*}\label{key}
	(1-\theta_k)\,\p_t \tu_k(t)  
	+
	\theta_k\,\tu_k(t)
	+ 
	D\E(u(t))
	= 0, \qquad k=1,\ldots,m.
	\end{equation*}
	These equations can be discretized in time and space with any preferred method.
	The history energy $\mathcal{H}$, see~\cref{eq:def_HC}, is then approximated with
	\begin{equation*}\label{key}
	\mathcal{H}(\tu(t)) = \frac{1}{2}\int_{0}^{1}\theta\,w_{\alpha}(\theta)\,\norm{\tu(t,\theta)}_H^2\dth\approx \frac12\sum_{k=1}^{m}w_k \theta_k \norm{\tu_{k}(t)}_H^2.
	\end{equation*}	

	However, in order to implement the method from~\cite{khristenko2021solving} based on rational approximations, we have first to reparameterize the integral of $\mathcal{C}\tu$, see~\cref{eq:u(tu)}.
	Using the change of variable $\lambda:=\theta/(1-\theta)$, it can be rewritten in the form
	\begin{equation*}\label{Eq:reparametrized}
	u(t)-u_0 
	= \frac{\sin(\pi\alpha)}{\pi}\int_{0}^{1}\theta^{-\alpha}(1-\theta)^{-\alpha}\,\tu(t,\theta)\dth
	=
	\frac{\sin(\pi\alpha)}{\pi}\int_{0}^{\infty}\frac{\lambda^{-\alpha}}{1+\lambda}\,\tu(t,\theta(\lambda))\dl.
	\end{equation*}
	Then, we use the following quadrature rule on the integral:
	\begin{equation}\label{Eq:NewQuad}
	u(t)
	\approx
	u_0 + \sum_{k=1}^{m}\frac{w_k}{1+\lambda_k} \tu_{k}(t) + w_\infty \tu_\infty(t)
	,
	\end{equation}
	where $\tu_{k}:=\tu(t,\theta(\lambda_k))$, the weights~$w_k$ and the nodes~$\lambda_k$ are given by the residues and the poles, respectively, in the rational approximation of the Laplace spectrum of the fractional kernel~$g_\alpha$, see~\cite{khristenko2021solving} for the details.
	In particular, we use the adaptive Antoulas--Anderson (AAA) algorithm~\cite[Fig. 4.1]{nakatsukasa2018aaa} to determine the rational approximation
	\begin{equation*}\label{key}
	z^{-\alpha} \approx \sum_{k=1}^{m}\frac{w_k}{z+\lambda_k}  + w_\infty, \qquad z\in[1/T,1/h].
	\end{equation*}
	Thus, the solution~$u$ is approximated by~\cref{Eq:NewQuad} with a linear combination of $m+1$ modes~$\tu_k=\tu(t,\theta_k)$, associated with the nodes $\theta_k=\lambda_k/(1+\lambda_k)$ and defined by the following system:
	\begin{equation*}\label{key}
	\begin{aligned}
	\frac{1}{1+\lambda_k}\pt \tu_k(t) + \frac{\lambda_k}{1+\lambda_k}\tu_k(t) + D\E(u(t)) &= 0,
	\qquad
	\qquad k=1,\ldots,m,\\
	\tu_{\infty}(t) + D\E(u(t)) &= 0.
	\end{aligned}	
	\end{equation*}	
	Accordingly, the history energy integral is reparametrized and approximated using the same quadrature rule:
	\begin{align*}\label{key}
	\mathcal{H}(\tu(t))
	&=
	\frac{1}{2}\,\frac{\sin(\pi\alpha)}{\pi}\int_{0}^{\infty}\frac{\lambda^{1-\alpha}}{(1+\lambda)^2}\,\norm{\tu(t,\theta(\lambda))}_H^2\dl
	\\
	&\approx 
	\frac12\sum_{k=1}^{m}\frac{w_k \lambda_k}{(1 + \lambda_k)^2} \norm{\tu_{k}(t)}_H^2 + \frac12 w_{\infty}\norm{\tu_{\infty}(t)}_H^2.
	\end{align*}

%% file: 4_simulation.tex
\section{Numerical example} \label{Sec:Sim}

In this section, we investigate the effect of the history part of the energy on the evolution of the Ginzburg--Landau energy which we have already seen in \cref{Ex:1}. It is given by the formula
\begin{equation*} \E(u)  = \int_\Omega \Psi(u) + \frac{\eps^2}{2} |\nabla u|^2 \dx.
	\end{equation*}
We take the underlying Hilbert space $H=\dot H^{-1}(\Omega)$ which yields the Cahn--Hilliard equation. The numerical results of this section have been obtained by implementing the procedure of the previous section in FEniCS \cite{fenics}. 	Moreover, for the discretization in time and space, we follow the convex-concave splitting scheme proposed in~\cite[Section 6.2]{khristenko2021solving} for the time-fractional Cahn-Hilliard equation, using standard mixed $Q1-Q1$ finite elements. 

\subsection{Simulation setup}
We apply the time-fractional Cahn--Hilliard equation to a phase separation process such as in the case of binary alloys. We assume zero source, homogeneous Neumann boundary, and we take a randomly distributed initial condition $u_0$ in the interval $[-10^{-3},10^{-3}]$. In order to ensure the phase-separation process, it is recommended \cite{deckelnick2005computation} to scale the prefactor of $\Psi$ with the interfacial width $\eps$. In particular, we introduce $\tilde{\eps},\beta>0$, define $\eps^2=\tilde \eps \beta$, and choose the prefactor of $\Psi$ as $\beta/\tilde{\eps}$. In the simulations, we select $\tilde \eps=0.05$, $\beta=0.1$, $M=1$. Consequently, the scaled Landau potential $\Psi(u)=\tfrac12(1-u^2)^2$ and the interface width
$\eps^2=0.005$ is considered.

For the simulation setup, we choose a domain $\Omega=(0,1)^2$ with the mesh size  $\Delta x=2^{-7}$ to ensure that the interface can be resolved and the condition $0.0078 \approx \Delta x \ll \tilde \eps=0.05$ holds, see \cite{deckelnick2005computation}. Concerning the time step size, we want to satisfy the condition for energy stability in case of the implicit scheme in the integer-order setting, see \cite{barrett1999finite}. The condition writes $\Delta t<4\eps^2/(M C_\Psi^2)$, where $C_\Psi$ is the prefactor of the potential. We employ the convex-concave splitting scheme which is unconditionally stable in the integer-order setting $\alpha=1$, see \cite{elliott1993global}. Nonetheless, the condition for the implicit scheme is a good indicator for the solution to behave well; we refer to the discussion in \cite{wu2014stabilized} and the included numerical simulations which show incorrect solution behavior for larger time steps regardless of the unconditional stability. In our case, it suggests $\Delta t<0.01$ and we choose $\Delta t = 0.005$ over the time interval $[0,5]$.

\subsection{Simulation result}
We plot the evolution of the Ginzburg--Landau energy $\E$ in \cref{Plot:Energya} for six values of the fractional exponent; namely $\alpha \in \{0.1,0.3,0.5,0.7,0.9,1.0\}$. Firstly, we observe that the energy decays monotonically for each $\alpha$. We want to point out once more that it is not known whether the energy of a time-fractional gradient flow is dissipating, see the discussions in \cite{zhang2020non,liu2020fast,liao2020second,chen2019accurate}.
Consequently, no counter example of an increasing energy is known. Nonetheless, we investigate the behavior of the history energy and the influence on the augmented energy which is naturally decaying. 

We observe in \cref{Plot:Energya} that smaller $\alpha$ admit an earlier energy drop of larger magnitude at the initial time than large values of~$\alpha$. E.g., for $\alpha=0.1$ the energy drops from $\E=0.5$ to $\E\approx 0.16$ at $t \approx 0.16$, whereas for $\alpha=0.9$ it goes from $\E=0.5$ to $\E \approx 0.3$ at $t \approx 0.24$. This corresponds to the instantaneous process of time-fractional PDEs, see also the numerical studies in~\cite{fritz2020subdiffusive}. Furthermore, we observe that the energy reaches its asymptotic regime at $\E\approx 0.12$ at different times for each $\alpha$, e.g., for $\alpha=0.1$ at $t \approx 1.25$ versus $t \approx 2.4$ for $\alpha=0.7$.

\begin{figure}[!htb] \centering
	 \begin{subfigure}[t]{0.485\textwidth} \centering	\includegraphics[page=1,clip,width=\textwidth,trim={0 .5cm 0 0}]{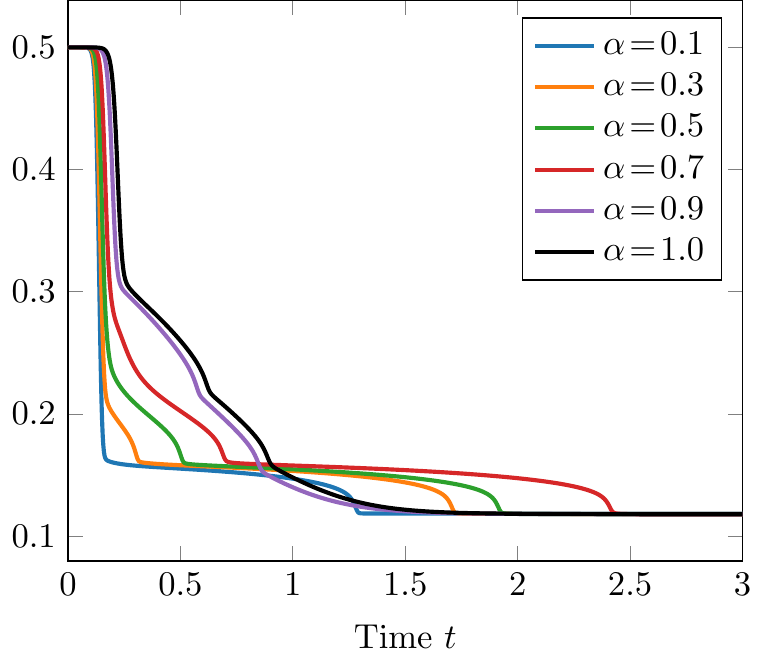}
	 	\caption{\label{Plot:Energya}Energy $\E$.} 
	\end{subfigure}	 \quad
	\begin{subfigure}[t]{0.485\textwidth} \centering	\includegraphics[clip,width=\textwidth,trim={0 .5cm 0 0}]{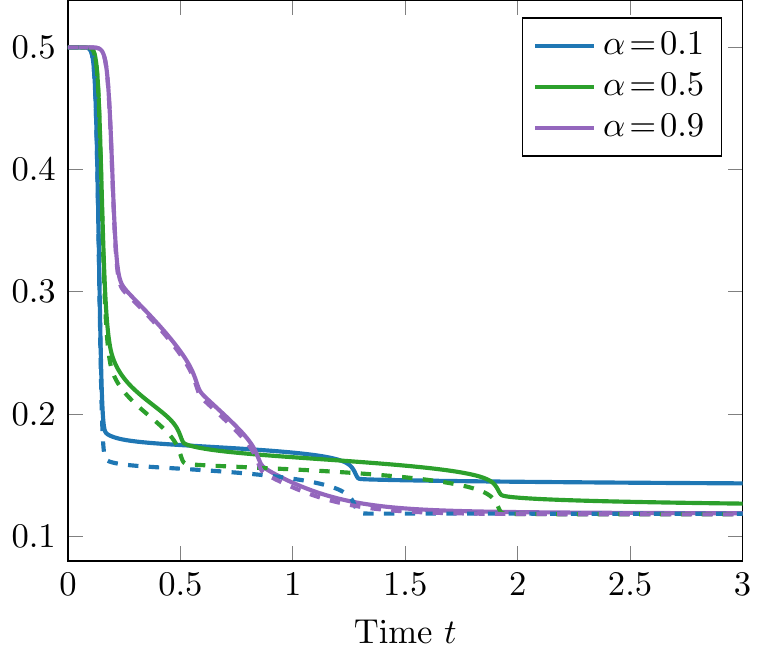}
		\caption{\label{Plot:Energyb}Energy $\E$ (dashed) and augmented energy $\E^\text{aug}$.} 
	\end{subfigure}	
			\caption{\label{Plot:Energy}
				Left: Evolution of the standard energy $\E$ for $\alpha \in \{0.1,0.3,0.5,0.7,0.9,1.0\}$.
				Right: Comparison of the augmented energy~$\E^\text{aug}$ and original energy $\E$ (dashed) for $\alpha \in \{0.1,0.5,0.9\}$.}
\end{figure}
\pagebreak 

In \cref{Plot:Energyb}, we plot the difference between the energy of the time-fractional gradient flow and the energy of the augmented system. We observe that the augmented energy deviates from the original energy from the first energy drop on, e.g., for $\alpha=0.1$ the energy drop is not as large for the augmented energy (from $\E=\E^\text{aug}=0.5$ to $\E^\text{aug} \approx 0.18$) as for the original one (to $\E\approx 0.16$). Consequently, it takes a longer time for the augmented energy to reach such an equilibrium-like state as the original energy. We interpret the plot as the augmented energy cushions the hard kink of the original energy and softens the curve shape. Rephrasing, one could say that the augmented energy avoids the places where the original energy is almost constant and runs the risk of being non-dissipative. %

We plot the history energy for $\alpha \in \{0.1,0.5,0.9,1.0\}$ in \cref{Plot:Historya}.
In the case of an integer-order gradient flow, $\alpha=1$, we observe the expected scenario -- the history energy is zero (see~\Cref{cor:vanish}).
That is, no memory effects are present, and thus the augmented and the standard energies are the same.
For the other values of $\alpha$, we observe several local maxima at the places where the energy admitted the kinks.
E.g., for $\alpha=0.1$ we had kinks in the energy in \cref{Plot:Energyb} at $t \approx 0.16$ and $t \approx 1.25$, whereas the history energy for $\alpha=0.1$ admits its local maxima at these points.
We also observe that smaller values of $\alpha$ lead to a larger history contribution.

In the asymptotic regime, as soon as the energy approaches its equilibrium state, the history part monotonically decays to zero (see also~\Cref{rmk:SteadyState}).
In order to estimate the asymptotic slope of the history energy, we plot it for $\alpha \in \{0.1,0.3,0.5\}$ on the time interval $[3,5]$ on a logarithmic scale in \cref{Plot:Historyb}, comparing to the slopes of $t^{-\beta}$ with $\beta \in \{0.11,0.36,0.65\}$ for each $\alpha$ respectively.
We observe a good fitting in the asymptotic regime of the history energy with a parameter $\beta$ which is close to the fractional exponent $\alpha$. Moreover, we observe numerically that the history part can be asymptotically bounded by $C(\alpha) t^{-\alpha}$ with $C(\alpha)$ tending to zero for $\alpha \to 1$.

\begin{figure}[!htb] \centering
	\begin{subfigure}[t]{0.46\textwidth} \centering	\includegraphics[width=\textwidth]{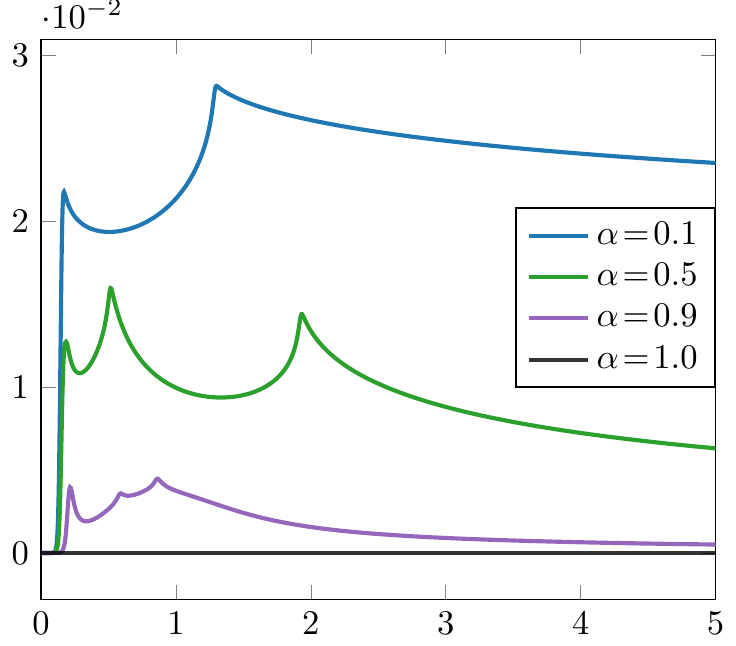}
	\caption{\label{Plot:Historya}History energy $\H$.} 
\end{subfigure}	
\begin{subfigure}[t]{0.495\textwidth} \centering	\includegraphics[width=\textwidth]{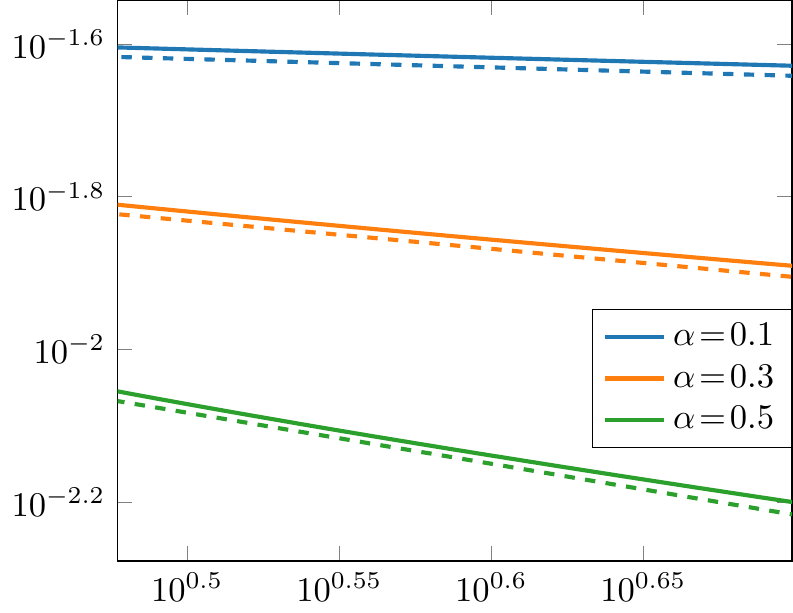}
	\caption{\label{Plot:Historyb}Asymptotic behavior of the history energy~$\H$.}
\end{subfigure}	
			\caption{\label{Plot:History}
			Left: The history energy $\H$ for $\alpha \in \{0.1,0.5,0.9,1.0\}$.
			Right: Asymptotic behavior of the history energy~$\H$ on the logarithmic scale. Dashed lines corresponds to the slope of $t^{-\beta}$ with $\beta \in \{0.11,0.36,0.65\}$, from top to bottom, respectively 
		}
\end{figure}